\documentclass[12pt]{article}
\usepackage{tikz}
\usetikzlibrary{calc,through,backgrounds}
\usepackage{multicol}

\usepackage{bm,amssymb,amsthm,amsfonts,amsmath,ytableau,latexsym,cite,color, psfrag,graphicx,ifpdf,pgfkeys,pgfopts,xcolor,extarrows}
\usepackage{graphicx}
\usepackage{subcaption}

\newtheorem{thm}{Theorem}[section]
\newtheorem{prop}[thm]{Proposition}
\newtheorem{conj}[thm]{Conjecture}
\newtheorem{lem}[thm]{Lemma}

\theoremstyle{definition}
\newtheorem{exmp}[thm]{Example}

\allowdisplaybreaks

\makeatletter \@addtoreset{equation}{section}
\@addtoreset{theo}{}\makeatother

\usepackage[pagebackref]{hyperref}

\def\S{  \mathfrak{S}}
\def\G{  \mathfrak{G}}

\def\Newton{  \mathrm{Newton}}
\def\BPD{  \mathrm{BPD}}
\def\supp{  \mathrm{supp}}

\begin{document}
\begin{center}
{\bf ZERO-ONE GROTHENDIECK POLYNOMIALS}

\vskip 4mm
{\small  YIMING CHEN, NEIL J.Y. FAN, AND ZELIN YE}

\end{center}

\noindent{A{\scriptsize BSTRACT}.}
Fink,  M\'esz\'aros and  St.\,Dizier showed that the Schubert polynomial $\S_w(x)$ is zero-one if and only if $w$ avoids twelve permutation patterns. In this paper, we prove that the  Grothendieck polynomial $\G_w(x)$ is zero-one, i.e., with coefficients either 0 or $\pm$1, if and only if $w$ avoids six patterns. As applications, we show that the normalized double Schubert  polynomial $N(\S_w(x;y))$ is Lorentzian when $\G_w(x)$ is zero-one, partially confirming a conjecture of Huh,  Matherne, M\'esz\'aros and St.\,Dizier.   
 Moreover, we verify several conjectures on the support and coefficients of  Grothendieck polynomials posed by M\'{e}sz\'{a}ros, Setiabrata and St.\,Dizier for the case of zero-one Grothendieck polynomials.

\section{Introduction}\label{sec:introduction}

Grothendieck polynomials were introduced by Lascoux and Sch\"{u}tzenberger \cite{LascouxSchützenberger1982Gro} as polynomial representatives for the K-theory classes of the structure sheaves of Schubert varieties. The lowest degree components of Grothendieck polynomials are Schubert polynomials, also introduced by  Lascoux and Sch\"{u}tzenberger \cite{LascouxSchützenberger1982}, which are polynomial representatives for Schubert classes in flag varieties.  These two families of polynomials serve as the bridge between the geometry and topology of Schubert varieties and the associated algebra and combinatorics, and have been studied extensively from various aspects. In particular, there are many combinatorial models to generate these two polynomials, see, for example, \cite{BergeronBilley1993,Magyar1998,LamLeeShimozono2018,Anna2021,FominKirillov1994}.

Let $[n]=\{1,2,\ldots,n\}$ and $S_n$ be the symmetric group on $[n]$. For $w\in S_n$, 
denote $\overline{X}_w$ by the matrix Schubert variety of $w$, which can be viewed as a subvariety of the product of projective spaces. 
Knutson and Miller \cite{KnutsonMiller2005} showed that the Schubert polynomial $\S_w(x)$ equals the multidegree polynomial of $\overline{X}_w$, and the
Grothendieck polynomial $\G_w(x)$ coincides with the twisted $K$-polynomial of $\overline{X}_w$. 
As a consequence,  the Schubert polynomial $\S_w(x)$ is zero-one, i.e., has coefficients 0 or 1, if and only if  $\overline{X}_w$   is a multiplicity-free variety. For the definition and properties of multiplicity-free varieties, see Brion \cite{Brion} and the references therein. 

Fink, M\'{e}sz\'{a}ros and St.\,Dizier \cite{FinkMeszarosDizier2020} showed that the Schubert polynomial $\S_{w}(x)$ is zero-one if and only if $w$ avoids the twelve patterns $12543$, $13254$, $13524$, $13542$,
$21543$, $125364$, $125634$, $215364$, $215634$, $315264$, $315624$ and $315642$. Recall that for $w=w(1)\cdots w(n)\in S_n$ and $\sigma=\sigma(1)\cdots\sigma(k)\in S_k$ with $k\leq n$,  if there exist  $i_1<i_2<\cdots<i_k$ such that $w(i_1),w(i_2),\ldots,w(i_k)$ have the same relative order as $\sigma(1),\sigma(2),\ldots,\sigma(k)$, then we say that $w$ contains   $\sigma$. Otherwise, we say that  $w$ avoids $\sigma$. 
In this paper, we show that

\begin{thm}\label{thm:main theorem}
The Grothendieck polynomial $\G_{w}(x)$ is zero-one, i.e., has coefficients  $0$ or $\pm1$, if and only if $w$ avoids the patterns $1432$, $1342$, $13254$, $31524$, $12534$ and $21534$. 
\end{thm}

As the first application of Theorem \ref{thm:main theorem}, we show that if $\G_w(x)$ is zero-one, then the normalized double Schubert polynomial $N(\S_w(x;y))$ has the Lorentzian property introduced by  Br\"{a}nd\'{e}n and Huh \cite{PetterHuh2020}, partially confirming a conjecture raised by Huh, Matherne, M\'{e}sz\'{a}ros and St.\,Dizier \cite{HuhMatherneMészárosDizier2022}.

To describe the conjecture in \cite{HuhMatherneMészárosDizier2022}, we need to recall some notation.
For a subset $J\subseteq\mathbb{Z}^{n}$,  $J$ is called \emph{M-convex} if for any $i\in [n]$ and any entries $\alpha,\beta\in J$ with $\alpha_i>\beta_i$, there is some $j\in [n]$ such that
$$\alpha_j<\beta_j\text{ and }\alpha-\epsilon_i+\epsilon_j\in J,$$
where $\epsilon_i$ is the vector with $1$ at the $i$-th position and $0$ at other positions. Given a polynomial $f=\sum_{\alpha\in \mathbb{Z}^{n}}c_{\alpha}x^{\alpha}$,  the \emph{support} of $f$ is 
$$\supp(f):=\{\alpha|c_{\alpha}\neq 0\}\subseteq \mathbb{Z}^{n}.$$
Suppose that  $f$ is a degree $d$ homogeneous polynomial with nonnegative coefficients. We say that $f$ is a \emph{Lorentzian polynomial} if $\supp(f)$ is M-convex and the quadratic  form
	$$\frac{\partial}{\partial x_{i_1}}\cdots \frac{\partial}{\partial x_{i_{d-2}}}f$$
	has at most one positive eigenvalue for any $i_1,\ldots,i_{d-2}\in [n]$.  
 The \emph{normalization operator} $N$ acting on monomials  is defined as
$$N(x^{\alpha})=\frac{x_1^{\alpha_1}}{\alpha_1!}\cdots \frac{x_m^{\alpha_m}}{\alpha_m!},$$
and extends linearly to polynomials.

\begin{conj}[\mdseries{\cite[Conjecture 15]{HuhMatherneMészárosDizier2022}}]\label{conj:Lorentzian property1sch}
For any $w\in S_n$, the polynomial $N(\S_w(x))$ is Lorentzian. 
\end{conj}

It was shown in \cite[Proposition 17]{HuhMatherneMészárosDizier2022} that if $w$ avoids the patterns 1432 and 1423, then $N(\S_w(x))$ is Lorentzian. Note that  Schubert polynomials $\S_w(x)$ reach their upper bounds if and only if $w$ avoids the patterns 1432 and 1423, see \cite{FanGuo2021}. 
 Castillo, Cid-Ruiz, Mohammadi, and Monta\~{n}o \cite{CastilloCidRuizMohammadiMontano2022} showed that if $X$ is a multiplicity-free variety, then the support of the $K$-polynomial of $X$ is a generalized polymatroid. In particular, this implies that if $\G_w(x)$ is zero-one, then the normalized homogeneous Grothendieck polynomial $N(\widehat{\G}_w(x,z))$ is Lorentzian, thus $N(\S_w(x))$ is Lorentzian.

It is expected in \cite{HuhMatherneMészárosDizier2022} that Conjecture \ref{conj:Lorentzian property1sch} also holds for double Schubert polynomials. We remark that even if $\G_w(x)$ is zero-one, the double Schubert polynomial $\S_w(x;y)$  is not necessarily zero-one any longer. 

\begin{thm}\label{thm:conjectures of Lorentzian property2}
If $\G_w(x)$ is zero-one, then $N(\widetilde{\S}_w(x;y))$ is Lorentzian, where  $\widetilde{\S}_w(x;y)$ is obtained from $\S_w(x;y)$ by changing all negative coefficients   to positive. 
\end{thm}

As the second application of Theorem \ref{thm:main theorem}, we show that the support and coefficients of a zero-one Grothendieck polynomial satisfy several conjectures proposed by M\'{e}sz\'{a}ros, Setiabrata and St.\,Dizier  \cite{MészárosSetiabrataDizier2022}.

Recall that the \emph{Newton polytope} of a polynomial $f$, denoted as 
$\Newton(f)$,  is
the convex hull of $\supp(f)$.  If
$$\Newton(f)\cap \mathbb{Z}^{n}=\supp(f),$$
then $f$ is said to have saturated Newton polytope, or SNP for short. For a vector $\alpha=(\alpha_1,\ldots,\alpha_n)$, let $|\alpha|=\alpha_1+\cdots+\alpha_n$. Define a partial order $\leq$ on the set $\supp(\G_w(x))$ by letting $\alpha\leq \beta$ if $\alpha_i\leq\beta_i$ for all $i$.

\begin{conj}[\mdseries{\cite[Conjecture 1.1]{MészárosSetiabrataDizier2022}}]\label{conj:support1}
	If $\alpha \in \supp(\G_w(x))$ and $|\alpha|< deg(\G_w(x))$, then there exists $\beta \in \supp(\G_w(x))$ with $\alpha < \beta$.
\end{conj}

A natural strengthening of Conjecture \ref{conj:support1} is

\begin{conj}[\mdseries{\cite[Conjecture 1.2]{MészárosSetiabrataDizier2022}}]\label{conj:support2}
If $\alpha \in \supp(\G_w(x))$ and $|\alpha|< deg(\G_w(x))$, then there exists $\beta \in \supp(\G_w(x))$ with $\alpha<\beta$ and $|\beta|=|\alpha|+1$.
\end{conj}

\begin{conj}[\mdseries{\cite[Conjecture 1.6]{MészárosSetiabrataDizier2022}}]\label{conj:coefficient}
	Fix $w\in S_n$ and let $\G_{w}(x)=\sum_{\alpha\in \mathbb{Z}^n}C_{w\alpha}x^{\alpha}$. For any $\beta\in \supp(\G_w^{top}(x))$, i.e.,  $x^{\beta}$ is a monomial of the highest degree in $\G_{w}(x)$, we have
	$$\sum_{\alpha\leq\beta}C_{w\alpha}=1.$$
\end{conj}

M\'{e}sz\'{a}ros, Setiabrata and St.\,Dizier  \cite{MészárosSetiabrataDizier2022} showed that Conjectures \ref{conj:support1} and \ref{conj:coefficient} hold for fireworks permutations.
Recently, Chou and Yu \cite{ChouandYu} proved  Conjecture \ref{conj:support2} for inverse fireworks permutations.

\begin{thm}\label{thm:conjectures of support and coefficient}
Conjectures \ref{conj:support1},  \ref{conj:support2} and \ref{conj:coefficient} hold for zero-one Grothendieck polynomials.  
\end{thm}

This paper is organized as follows. In Section \ref{sec:Preliminary}, we recall basic definitions of Schubert and Grothendieck polynomials as well as the bumpless pipe dream model. In Section \ref{sec:proof of main theorem}, we give a proof of   Theorem \ref{thm:main theorem}  and obtain a factorization of zero-one Grothendieck polynomials. In Section \ref{sec:Properties of zero-one Grothendieck polynomials}, we prove Theorem \ref{thm:conjectures of Lorentzian property2} and Theorem \ref{thm:conjectures of support and coefficient}.

\vspace{.2cm} \noindent{\bf Acknowledgments.}
Part of  this work was carried out during the PACE program in the summer of 2023 at BICMR in Peking
University. We thank the program and the university for providing an inspiring
 atmosphere for research.
We would also like to thank Yibo Gao, Peter L. Guo and Rui Xiong for helpful discussions. 
This work was
supported by the National Science Foundation of China (12071320) and Sichuan Science and Technology Program (2023ZYD0012).

\section{Preliminaries}\label{sec:Preliminary}
In this section, we first recall definitions of Schubert polynomials and Grothendieck polynomials. Then we explain the bumpless pipe dream (BPD) model for Schubert and Grothendieck polynomials in \cite{LamLeeShimozono2018} \cite{Anna2021}.

For $w=w(1)\cdots w(n)\in S_n$, denote $\ell(w)$ the length of $w$. Let $w_0=n(n-1)\cdots 21$ be the longest permutation in $S_n$.  
Let $f$ be a polynomial in variables $x_1,\ldots,x_n,y_1,\ldots,y_n$, the divided difference operator $\partial_{i}$ on $f$ is defined as
	$$\partial_{i}f=\frac{f-s_if}{x_i-x_{i+1}},$$
    where $s_{i}f$ interchanges $x_i$ and $x_{i+1}$ in $f$. The isobaric divided difference  operator $\pi_{i}$ on $f$ is defined as
$$\pi_{i}f=\frac{(1-x_{i+1})f-(1-x_i)s_{i}f}{x_i-x_{i+1}}.$$

The \emph{double Schubert polynomial} $\S_{w}(x;y)$ is defined recursively as follows. For the longest permutation $w_0\in S_n$, set
$$\S_{w_0}(x;y)=\prod_{i+j\le n}(x_{i}-y_{j}).$$
If $w(i)<w(i+1)$ for some $i$, then let
$$\S_{w}(x;y)=\partial_{i}\S_{ws_i}(x;y),$$
where $ws_i$ is obtained from $w$ by interchanging $w(i)$ and $w(i+1)$. Setting $y=0$ in $\S_{w}(x;y)$, we obtain the single Schubert polynomial $\S_w(x).$

The \emph{double Grothendieck polynomial} $\G_{w}(x;y)$ for the longest permutation $w_0\in S_n$ is defined to be 
$$\G_{w_0}(x;y)=\prod_{i+j\le n}(x_{i}+y_{j}-x_{i}y_{j}),$$
and if $w(i)<w(i+1)$ for some $i$, then 
$$\G_{w}(x;y)=\pi_{i}\G_{ws_i}(x;y).$$
Letting $y=0$ in $\G_{w}(x;y)$, we obtain the single Grothendieck polynomial $\G_{w}(x).$

It is well known that
Schubert polynomials are homogeneous with degree $\ell(w)$. The lowest degree homogeneous component of $\G_{w}(x)$ is the Schubert polynomial $\S_{w}(x)$. For our purpose, it is more convenient to change all coefficients of $\G_{w}(x)$ to positive. Thus we define
$$\widetilde{\G}_{w}(x):=\sum_{k=0}^{d(w)-\ell(w)}(-1)^{k}\G_{w}^{(k)}(x),$$
where $\G_{w}^{(k)}(x)$ is the degree $k+\ell(w)$  component of $\G_{w}(x),$ and $d(w)$ is the degree of $\G_w(x)$, which is exactly the statistic ${\rm raj}(w)$ introduced by Pechenik, Speyer and Weigandt  \cite{PechenikSpeyerWeigandt2021}.  Moreover,  the \emph{homogeneous Grothendieck polynomial} of $w$ is defined by
$$\widehat{\G}_w(x,z):=\sum_{k=0}^{d(w)-\ell(w)}(-1)^{k}\G_{w}^{(k)}(x)z^{d(w)-\ell(w)-k},$$
where $z$ is a new variable.

In the following, we recall the bumpless pipe dream model for Schubert polynomials introduced by Lam, Lee and Shimozono \cite{LamLeeShimozono2018} and extended to Grothendieck polynomials by Weigandt \cite{Anna2021}. 

For a permutation $w\in S_n$, a \emph{bumpless pipe dream}  for $w$ is an $n\times n$ square grid with $n$ pipes labeled $1,2,\ldots,n$. Each pipe satisfies the following conditions:
\begin{enumerate}

\item[(1)] Pipe $i$ starts from the south boundary  in column $i$ (counted from left to right) and ends at the east boundary  in row $w^{-1}(i)$  (counted from top to bottom). Pipes can only go up or right.

\item[(2)]  Every box of the $n\times n$ grid looks like one of the six tiles shown in Figure \ref{fig:six tiles of BPD}. We call them from left to right: an SE elbow, an NW elbow, a horizontal tile, a vertical tile, a crossing and an empty box, respectively. In particular,  no two pipes can change their directions in one box. Such a box is called a bump \raisebox{-0.35\height}{
\begin{tikzpicture}[scale=0.6]
\draw (0,0) rectangle (1,1);
\draw[thick, blue, rounded corners=0.15cm] (0.5,0)--(0.5,0.5)--(1,0.5);
\draw[thick, blue, rounded corners=0.15cm] (0,0.5)--(0.5,0.5)--(0.5,1);
\end{tikzpicture}
}. 

\begin{figure}[ht]
	\centering
	\begin{tikzpicture}[scale=1]
	\draw (0,0) rectangle (1,1);
	\draw (2,0) rectangle (3,1);
	\draw (4,0) rectangle (5,1);
	\draw (6,0) rectangle (7,1);
	\draw (8,0) rectangle (9,1);
	\draw (10,0) rectangle (11,1);
	\draw[thick, blue, rounded corners=0.15cm] (0.5,0)--(0.5,0.5)--(1,0.5);
	\draw[thick, blue, rounded corners=0.15cm] (2,0.5)--(2.5,0.5)--(2.5,1);
	\draw[thick, blue] (4,0.5)--(5,0.5);
	\draw[thick, blue] (6.5,0)--(6.5,1);
	\draw[thick, blue] (8.5,0)--(8.5,1);
	\draw[thick, blue] (8,0.5)--(9,0.5);
	\end{tikzpicture}
	\caption{The six tiles of BPD}
	\label{fig:six tiles of BPD}
\end{figure}
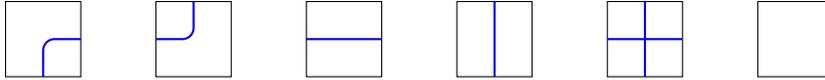

\item[(3)] No two pipes can overlap. If two pipes cross more than once, then   all the crossings after the first crossing (counted from left to right) are interpreted as bumps.

\end{enumerate}

Given a permutation $w\in S_n$, denote $\BPD(w)$ by the set of bumpless pipe dreams of $w$. 
For ease of drawing pictures, we will replace an SE elbow    \raisebox{-0.35\height}{
\begin{tikzpicture}[scale=0.6]
\draw (0,0) rectangle (1,1);
	\draw[thick, blue, rounded corners=0.15cm] (0.5,0)--(0.5,0.5)--(1,0.5);
\end{tikzpicture}
} by \raisebox{-0.35\height}{
\begin{tikzpicture}[scale=0.6]
\draw (0,0) rectangle (1,1);
\draw[thick, blue] (0.5,0)--(0.5,0.5)--(1,0.5);
\end{tikzpicture}
}, and an NW elbow \raisebox{-0.35\height}{
\begin{tikzpicture}[scale=0.6]
\draw (0,0) rectangle (1,1);
\draw[thick, blue, rounded corners=0.15cm] (0,0.5)--(0.5,0.5)--(0.5,1);
\end{tikzpicture}
} by \raisebox{-0.35\height}{
\begin{tikzpicture}[scale=0.6]
\draw (0,0) rectangle (1,1);
\draw[thick, blue] (0,0.5)--(0.5,0.5)--(0.5,1);
\end{tikzpicture}
}. Since no two pipes can bump in one box, there is no ambiguity for this replacement.
For example, let $w=1342$. Then  $\BPD(w)$ is displayed in Figure \ref{fig:BPD(1342)}, where the column labels indicate the row indices of right ends of the pipes. 
\begin{figure}[ht]
	\centering
	\begin{minipage}{0.25\textwidth}
			\begin{tikzpicture}[scale=0.7]
            \draw[thin] (0,0) rectangle (4,4);
			\draw[very thin, step=1, dashed] (0,0) grid (4,4); 
			\draw[thick, blue](0.5,0)--(0.5,3.5)--(4,3.5);
			\draw[thick, blue](1.5,0)--(1.5,0.5)--(4,0.5);
			\draw[thick, blue](2.5,0)--(2.5,2.5)--(4,2.5);
			\draw[thick, blue](3.5,0)--(3.5,1.5)--(4,1.5);
			\node[below] at (0.5,0) {$1$};
			\node[below] at (1.5,0) {$2$};
			\node[below] at (2.5,0) {$3$};
			\node[below] at (3.5,0) {$4$};
			\node[right] at (4,0.5) {\footnotesize{$w^{-1}(2)$}};
			\node[right] at (4,1.5) {\footnotesize{${w^{-1}(4)}$}};
			\node[right] at (4,2.5) {\footnotesize{$w^{-1}(3)$}};
			\node[right] at (4,3.5) {\footnotesize{$w^{-1}(1)$}};
			\end{tikzpicture}
		\end{minipage}
		\quad
		\begin{minipage}{0.25\textwidth}
			\begin{tikzpicture}[scale=0.7]
            \draw[thin] (0,0) rectangle (4,4);
			\draw[very thin, step=1, dashed] (0,0) grid (4,4); 
			\draw[thick, blue](0.5,0)--(0.5,2.5)--(1.5,2.5)--(1.5,3.5)--(4,3.5);
			\draw[thick, blue](1.5,0)--(1.5,0.5)--(4,0.5);
			\draw[thick, blue](2.5,0)--(2.5,2.5)--(4,2.5);
			\draw[thick, blue](3.5,0)--(3.5,1.5)--(4,1.5);
			\node[below] at (0.5,0) {$1$};
			\node[below] at (1.5,0) {$2$};
			\node[below] at (2.5,0) {$3$};
			\node[below] at (3.5,0) {$4$};
			\node[right] at (4,0.5) {\footnotesize{$w^{-1}(2)$}};
			\node[right] at (4,1.5) {\footnotesize{${w^{-1}(4)}$}};
			\node[right] at (4,2.5) {\footnotesize{$w^{-1}(3)$}};
			\node[right] at (4,3.5) {\footnotesize{$w^{-1}(1)$}};
			\end{tikzpicture}
		\end{minipage}
		\quad
		\begin{minipage}{0.25\textwidth}
			\begin{tikzpicture}[scale=0.7]
            \draw[thin] (0,0) rectangle (4,4);
			\draw[very thin, step=1, dashed] (0,0) grid (4,4); 
			\draw[thick, blue](0.5,0)--(0.5,1.5)--(1.5,1.5)--(1.5,3.5)--(4,3.5);
			\draw[thick, blue](1.5,0)--(1.5,0.5)--(4,0.5);
			\draw[thick, blue](2.5,0)--(2.5,2.5)--(4,2.5);
			\draw[thick, blue](3.5,0)--(3.5,1.5)--(4,1.5);
			\node[below] at (0.5,0) {$1$};
			\node[below] at (1.5,0) {$2$};
			\node[below] at (2.5,0) {$3$};
			\node[below] at (3.5,0) {$4$};
			\node[right] at (4,0.5) {\footnotesize{$w^{-1}(2)$}};
			\node[right] at (4,1.5) {\footnotesize{${w^{-1}(4)}$}};
			\node[right] at (4,2.5) {\footnotesize{$w^{-1}(3)$}};
			\node[right] at (4,3.5) {\footnotesize{$w^{-1}(1)$}};
			\end{tikzpicture}
		\end{minipage}
		\caption{$\BPD(1342)$}
		\label{fig:BPD(1342)}
\end{figure}
 
Now we  recall the droop operations in \cite{LamLeeShimozono2018} and \cite{Anna2021}, which can generate all $\BPD(w)$ from the Rothe pipe dream of $w$. 
Recall that the \emph{Rothe diagram} of $w$ is
$$D(w)=\{(i,j):w(i)>j,w^{-1}(j)>i\}.$$
For $i=1,\ldots,n$, put a dot at  $(i,w(i))$,  draw a horizontal line to the right and a vertical line to the bottom from $(i,w(i))$. Then we obtain $n$ pipes satisfying the requirements of bumpless pipe dreams. We call it the \emph{Rothe pipe dream} of $w$. By definition, the empty boxes of the Rothe pipe dream of $w$ consist of the Rothe diagram $D(w)$.   By abuse of notation, we denote $D(w)$ by the Rothe diagram as well as the Rothe pipe dream of $w.$

Suppose that there is a rectangle region with an SE elbow in the northwest corner, an empty box in the southeast corner, and  no SE or NW elbows inside. A \emph{droop} operation first changes the SE elbow to an empty box, and the empty box to an NW elbow, then draws a horizontal line to the left and a vertical line to the top at the new NW elbow until these lines reach the original pipes, as is shown in Figure \ref{fig:droop}. 

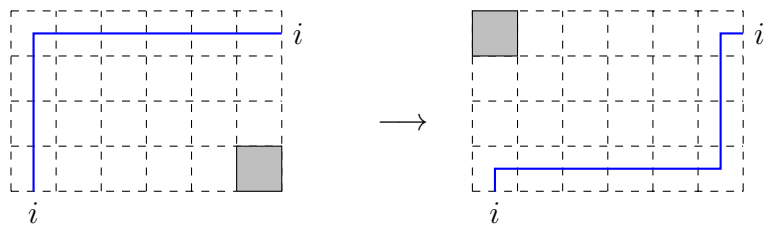
\begin{figure}[ht]
	\centering
	\begin{minipage}{0.3\textwidth}
		\begin{tikzpicture}[scale=0.6]
		\draw[dashed, step=1] (0,0) grid (6,4);
		\draw[fill=lightgray](5,0) rectangle (6,1);
		\draw[thick, blue](0.5,0)--(0.5,3.5)--(6,3.5);
		\node[below] at (0.5,0) {$i$};
		\node[right] at (6,3.5) {$i$};
		\end{tikzpicture}
	\end{minipage}
    $\longrightarrow$
    \hspace{.3cm}
	\begin{minipage}{0.3\textwidth}
		\begin{tikzpicture}[scale=0.6]
		\draw[dashed, step=1] (0,0) grid (6,4);
		\draw[fill=lightgray](0,3) rectangle (1,4);
		\draw[thick, blue](0.5,0)--(0.5,0.5)--(5.5,0.5)--(5.5,3.5)--(6,3.5);
		\node[below] at (0.5,0) {$i$};
		\node[right] at (6,3.5) {$i$};
		\end{tikzpicture}
	\end{minipage}
	\caption{A droop operation}
	\label{fig:droop}
\end{figure}

Suppose now there are two pipes labeled $i$ and $j$ intersecting once at the northeast corner or southwest corner of a rectangle region. Pipe $i$ has an SE elbow at the northwest corner and pipe $j$ has an NW elbow at the southeast corner. The \emph{$K$-theoretic droop} operation has two situations displayed in Figure \ref{fig:$K$-theoretic droop}, which makes
pipe $i$ and pipe $j$ intersect twice and the second crossings are interpreted as bumps.

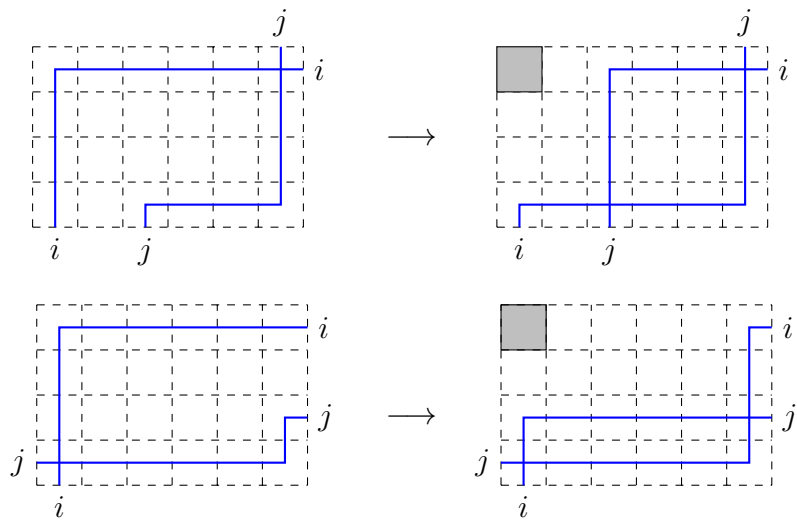
\begin{figure}[ht]
	\centering
	\begin{minipage}{0.3\textwidth}
		\begin{tikzpicture}[scale=0.6]
		\draw[dashed, step=1] (0,0) grid (6,4);
		\draw[thick, blue](0.5,0)--(0.5,3.5)--(6,3.5);
		\draw[thick, blue](2.5,0)--(2.5,0.5)--(5.5,0.5)--(5.5,4);
		\node[below] at (0.5,0) {$i$};
		\node[right] at (6,3.5) {$i$};
		\node[below] at (2.5,0) {$j$};
		\node[above] at (5.5,4) {$j$};
		\node[left] at (0,0.5) {$\  $};
		\end{tikzpicture}
	\end{minipage}
    \hspace{.1cm}
	$\longrightarrow$
    \hspace{.1cm}
	\begin{minipage}{0.3\textwidth}
		\begin{tikzpicture}[scale=0.6]
		\draw[dashed, step=1] (0,0) grid (6,4);
        \draw[fill=lightgray](0,3) rectangle (1,4);
		\draw[thick, blue](0.5,0)--(0.5,0.5)--(5.5,0.5)--(5.5,4);
		\draw[thick, blue](2.5,0)--(2.5,3.5)--(6,3.5);
		\node[below] at (0.5,0) {$i$};
		\node[right] at (6,3.5) {$i$};
		\node[below] at (2.5,0) {$j$};
		\node[above] at (5.5,4) {$j$};	
		\node[left] at (0,0.5) {$\  $};
		\end{tikzpicture}		
	\end{minipage}
 
\vspace{.35cm}

    \begin{minipage}{0.3\textwidth}
		\begin{tikzpicture}[scale=0.6]
		\draw[dashed, step=1] (0,0) grid (6,4);
		\draw[thick, blue](0.5,0)--(0.5,3.5)--(6,3.5);
		\draw[thick, blue](0,0.5)--(5.5,0.5)--(5.5,1.5)--(6,1.5);
		\node[below] at (0.5,0) {$i$};
		\node[right] at (6,3.5) {$i$};
		\node[left] at (0,0.5) {$j$};
		\node[right] at (6,1.5) {$j$};
		\end{tikzpicture}
	\end{minipage}
    \hspace{.1cm}
	$\longrightarrow$
    \hspace{.1cm}
	\begin{minipage}{0.3\textwidth}
		\begin{tikzpicture}[scale=0.6]
        \draw[fill=lightgray](0,3) rectangle (1,4);
		\draw[dashed, step=1] (0,0) grid (6,4);
		\draw[thick, blue](0.5,0)--(0.5,1.5)--(6,1.5);
		\draw[thick, blue](0,0.5)--(5.5,0.5)--(5.5,3.5)--(6,3.5);
		\node[below] at (0.5,0) {$i$};
		\node[right] at (6,3.5) {$i$};
		\node[left] at (0,0.5) {$j$};
		\node[right] at (6,1.5) {$j$};
		\end{tikzpicture}		
	\end{minipage}
	\caption{$K$-theoretic droop operations}
	\label{fig:$K$-theoretic droop}
\end{figure}

For both droops and $K$-theoretic droops, if there exist SE or NW elbows inside the involved rectangle, then it is not allowed to apply the droop operation, we call the droop is \emph{impermissible} in this situation. Otherwise, we say the droop is \emph{permissible}. 

\begin{prop}[\mdseries{\cite[Proposition 4.5]{Anna2021}}] \label{prop:Anna BPD can be obtained by droops}
	Any $\mathcal{P}\in \BPD(w)$ can be obtained from its Rothe pipe dream by a sequence of droops and $K$-theoretic droops.  
\end{prop}

Weigandt \cite{Anna2021} extended the formula for $\S_w(x;y)$ given by Lam, Lee and Shimozono \cite{LamLeeShimozono2018} to  double Grothendieck polynomials.

\begin{thm}[\mdseries{\cite{Anna2021}}]
For $w\in S_n$, we have 
\begin{align}
    \G_{w}(x;y)&=\sum_{\mathcal{P}\in \BPD(w)}\left(\prod_{(i,j)\in B(\mathcal{P})}(x_i+y_{j}-x_{i}y_{j}) \right) \left(\prod_{(p,q)\in U(\mathcal{P})}(1-x_{p}-y_{q}+x_{p}y_{q}) \right),\label{eq:doubleG}\\
    \G_{w}(x)&=\sum_{\mathcal{P}\in \BPD(w)}\left(\prod_{(i,j)\in B(\mathcal{P})}x_i \right) \left(\prod_{(p,q)\in U(\mathcal{P})}(1-x_{p}) \right),\label{eq:singG}
\end{align}
where
\begin{align*}
 B(\mathcal{P})&:=\{(i,j):\mathcal{P}\text{ has an empty box in row }i\text{ and column }j\},\\
U(\mathcal{P})&:=\{(p,q):\mathcal{P}\text{ has an NW elbow in row }p\text{ and column }q\}.   
\end{align*} 

\end{thm}

For the running example $w=1342$, the right two BPDs in Figure \ref{fig:BPD(1342)} are obtained by drooping the pipe labeled $1$ to the empty boxes at $(2,2)$ and $(3,2)$, respectively. Thus the double Grothendieck polynomial for $w=1342$ is
\begin{align*}
\G_{1342}(x;y)&=(x_2+y_2-x_2y_2)(x_3+y_2-x_3y_2)\\
    &\quad+(x_1+y_1-x_1y_1)(1-x_2-y_2+x_2y_2)(x_3+y_2-x_3y_2)\\
    &\quad+(x_1+y_1-x_1y_1)(x_2+y_1-x_2y_1)(1-x_3-y_2+x_3y_2).
\end{align*}

\section{Proof of the  Theorem \ref{thm:main theorem} }\label{sec:proof of main theorem}

In this section, we give a proof of Theorem \ref{thm:main theorem}. For the necessity, we construct monomials that appear more than once when $w$ contains one of the patterns $1432$, $1342$, $13254$, $31524$, $12534$ and $21534$. For the sufficiency, we show that $\widetilde{\G}_{w}(x)$ admits a factorization when $w$ avoids the above 6 patterns. This factorization of $\widetilde{\G}_{w}(x)$ implies the sufficiency of Theorem \ref{thm:main theorem} immediately and plays an important role in Section \ref{sec:Properties of zero-one Grothendieck polynomials}.

We first prove the necessity of Theorem \ref{thm:main theorem}.
 
\begin{thm}\label{thm:Necessity}
	If $\widetilde{\G}_{w}(x)$ is zero-one, then $w$  avoids the patterns $1432$, $1342$, $13254$, $31524$, $12534$ and $21534$. 
\end{thm}

\begin{proof}
We aim to show that if $w$ contains one of the 6 patterns, then there exists a monomial in $\widetilde{\G}_{w}(x)$ with its coefficient larger than $1$.

By \cite{FinkMeszarosDizier2020}, if $w$ contains the pattern $13254$, then its Schubert polynomial $\S_{w}(x)$  is already not zero-one. Thus $\G_{w}(x)$ is not zero-one. 
For the rest of 5 patterns, we are going to construct two different pipe dreams $\mathcal{P}_1,\mathcal{P}_2\in \BPD(w)$ such that $\mathcal{P}_1$ and $\mathcal{P}_2$ generate the same monomial by \eqref{eq:singG}. We divide the discussions into 4 cases.

\vspace{.2cm}
\noindent\textbf{Case 1.}  $w$ contains the pattern $1432$ or 1342. 
If $w$ contains the pattern $1432$, then there exist   $a<b<c<d$  such that $w(a)w(b)w(c)w(d)$ has the same relative order with 1432.  Let 
\[i_1=w(a),\ i_2=w(d),\ i_3=w(c),\ i_4=w(b).\] 
Then $i_1<i_2<i_3<i_4$ and
\[w^{-1}(i_1)<w^{-1}(i_4)<w^{-1}(i_3)<w^{-1}(i_2).\] 
 Similarly, if $w$ contains the pattern $1342$, then  there exsist $i_1<i_2<i_3<i_4$ such that
$w^{-1}(i_1)<w^{-1}(i_3)<w^{-1}(i_4)<w^{-1}(i_2),$ see the illustration in Figure \ref{fig:1432 and 1342}. 
By definition,   $(w^{-1}(i_3),i_2)$, $(w^{-1}(i_4),i_2)\in D(w)$.

\begin{figure}[ht]
	\centering
	\begin{minipage}{0.35\textwidth}
		\begin{tikzpicture}[scale=0.4]
		\draw[very thin, step=1] (0,0) rectangle (8,8); 
		\draw[thick, blue](0.5,0)--(0.5,7)--(8,7);
		\draw[thick, blue](2.5,0)--(2.5,1.5)--(8,1.5);
		\draw[thick, blue](5.5,0)--(5.5,2.5)--(8,2.5);
		\draw[thick, blue](7,0)--(7,5.5)--(8,5.5);
		\draw[fill=lightgray](2,2) rectangle (3,3);
		\draw[fill=lightgray](5,5) rectangle (6,6);
		\draw[fill=lightgray](2,5) rectangle (3,6);
		\node[below] at (0.5,0) {$i_1$};
		\node[below] at (2.5,0) {$i_2$};
		\node[below] at (5.5,0) {$i_3$};
		\node[below] at (7,0) {$i_4$};
		\node[right] at (8,7) {\footnotesize{$w^{-1}(i_1)$}};
		\node[right] at (8,5.5) {\footnotesize{$w^{-1}(i_4)$}};
		\node[right] at (8,2.5) {\footnotesize{$w^{-1}(i_3)$}};
		\node[right] at (8,1.5) {\footnotesize{$w^{-1}(i_2)$}};
		\end{tikzpicture}
	\end{minipage}
	\qquad
	\begin{minipage}{0.35\textwidth}
		\begin{tikzpicture}[scale=0.4]
		\draw[very thin, step=1] (0,0) rectangle (8,8); 
		\draw[thick, blue](0.5,0)--(0.5,7)--(8,7);
		\draw[thick, blue](2.5,0)--(2.5,1.5)--(8,1.5);
		\draw[thick, blue](5.5,0)--(5.5,5.5)--(8,5.5);
		\draw[thick, blue](7,0)--(7,2.5)--(8,2.5);
		\draw[fill=lightgray](2,5) rectangle (3,6);
		\draw[fill=lightgray](2,2) rectangle (3,3);
		\node[below] at (0.5,0) {$i_1$};
		\node[below] at (2.5,0) {$i_2$};
		\node[below] at (5.5,0) {$i_3$};
		\node[below] at (7,0) {$i_4$};
		\node[right] at (8,7) {\footnotesize$w^{-1}(i_1)$};
		\node[right] at (8,5.5) {\footnotesize$w^{-1}(i_3)$};
		\node[right] at (8,2.5) {\footnotesize$w^{-1}(i_4)$};
		\node[right] at (8,1.5) {\footnotesize$w^{-1}(i_2)$};
		\end{tikzpicture} 	
	\end{minipage}
	\caption{The patterns $1432$ and $1342$}
	\label{fig:1432 and 1342}
\end{figure}
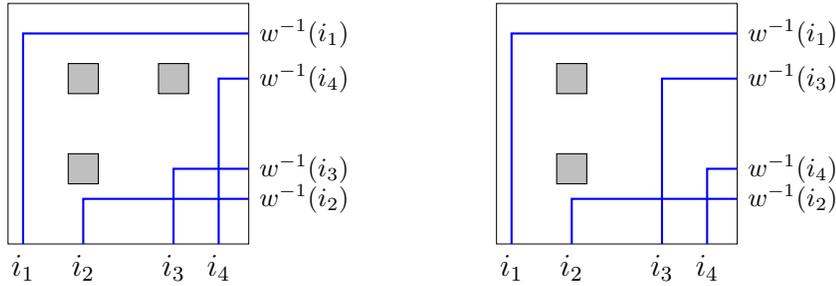

\vspace{.2cm}
\noindent\textbf{Claim 1:} If $w$ contains the pattern $1432$, then in the southeast region of pipe $i_1$, after a sequence of droops, there exist two empty boxes in the same column such that we can droop pipe $i_1$ to these two empty boxes. 

\begin{proof}[Proof of Claim 1]
If we can directly droop  pipe $i_1$ to the two empty boxes located at $(w^{-1}(i_3),i_2)$, $(w^{-1}(i_4),i_2)\in D(w)$ then there is nothing  to prove. Otherwise, it is not permissible to apply a droop in the region $[w^{-1}(i_1),w^{-1}(i_3)]\times [i_1,i_2]$, which implies there exist SE elbows in this region (Notice that there is no NW elbow in Rothe pipe dream).  There are two situations.
	
	\textcircled{1} If there is an SE elbow labeled  $i_1'$ in the region $[w^{-1}(i_1),w^{-1}(i_4)]\times[i_1,i_2]$, view  pipe $i_1'$  as  pipe $i_1$. 
	
\textcircled{2} If there is an SE elbow in the region $[w^{-1}(i_4),w^{-1}(i_3)]\times[i_1,i_2]$, then there exists a pipe $j$ such that $i_1<j<i_2$ and $w^{-1}(i_4)<w^{-1}(j)<w^{-1}(i_3)$, as shown in Figure \ref{fig:When droop is not permissible in pattern 1432}. Obviously, $(w^{-1}(i_4),j)\in D(w)$ is an empty box. If we can droop  pipe $j$ to the empty box $(w^{-1}(i_3),i_2)$, then we obtain two empty boxes $(w^{-1}(i_4),j)$ and $(w^{-1}(j),j)$ in the same column as desired. 

\begin{figure}[ht]
	\centering
    \begin{minipage}{0.3\textwidth}
    \begin{tikzpicture}[scale=0.4]
	\draw[very thin, step=1] (0,0) rectangle (8,8); 
	\draw[thick, blue](0.5,0)--(0.5,7)--(8,7);
	\draw[thick, blue](3.5,0)--(3.5,1.5)--(8,1.5);
	\draw[thick, blue](5.5,0)--(5.5,2.5)--(8,2.5);
	\draw[thick, red](1.5,0)--(1.5,4)--(8,4);
	\draw[thick, blue](7,0)--(7,5.5)--(8,5.5);
    \draw[draw=red,dashed] (1,4.5) rectangle (4,2);
	\draw[fill=lightgray](3,2) rectangle (4,3);
	\draw[fill=lightgray](5,5) rectangle (6,6);
	\draw[fill=lightgray](3,5) rectangle (4,6);
    \draw[fill=lightgray](1,5) rectangle (2,6);
	\node[below] at (0.5,0) {$i_1$};
	\node[below] at (3.5,0) {$i_2$};
	\node[below] at (5.5,0) {$i_3$};
	\node[below] at (7,0) {$i_4$};
	\node[below] at (1.5,0) {$j$};
	\node[right] at (8,7) {\footnotesize$w^{-1}(i_1)$};
	\node[right] at (8,5.5) {\footnotesize$w^{-1}(i_4)$};
	\node[right] at (8,2.5) {\footnotesize$w^{-1}(i_3)$};
	\node[right] at (8,1.5) {\footnotesize$w^{-1}(i_2)$};
	\node[right] at (8,4) {\footnotesize$w^{-1}(j)$};
	\end{tikzpicture}
    \end{minipage}
$\longrightarrow$
\quad
\begin{minipage}{0.35\textwidth}
    \begin{tikzpicture}[scale=0.4]
	\draw[very thin, step=1] (0,0) rectangle (8,8); 
	\draw[thick, blue](0.5,0)--(0.5,7)--(8,7);
	\draw[thick, blue](3.5,0)--(3.5,1.5)--(8,1.5);
	\draw[thick, blue](5.5,0)--(5.5,2.5)--(8,2.5);

	\draw[thick, blue](7,0)--(7,5.5)--(8,5.5);
    \draw[fill=lightgray](1,3.5) rectangle (2,4.5);
	\draw[fill=lightgray](5,5) rectangle (6,6);
	\draw[fill=lightgray](3,5) rectangle (4,6);
    \draw[fill=lightgray](1,5) rectangle (2,6);
    \draw[thick, red](1.5,0)--(1.5,2.5)--(3.5,2.5)--(3.5,4)--(8,4);
	\node[below] at (0.5,0) {$i_1$};
	\node[below] at (3.5,0) {$i_2$};
	\node[below] at (5.5,0) {$i_3$};
	\node[below] at (7,0) {$i_4$};
	\node[below] at (1.5,0) {$j$};
	\node[right] at (8,7) {\footnotesize{$w^{-1}(i_1)$}};
	\node[right] at (8,5.5) {\footnotesize$w^{-1}(i_4)$};
	\node[right] at (8,2.5) {\footnotesize$w^{-1}(i_3)$};
	\node[right] at (8,1.5) {\footnotesize$w^{-1}(i_2)$};
	\node[right] at (8,4) {\footnotesize$w^{-1}(j)$};
	\end{tikzpicture}
\end{minipage}
	\caption{An impermissible droop in the pattern $1432$}
	\label{fig:When droop is not permissible in pattern 1432}
\end{figure}
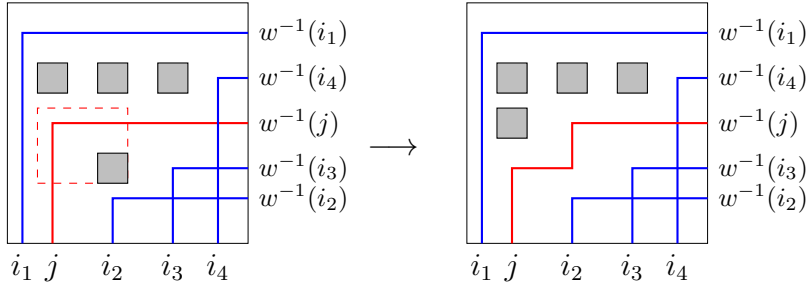

Otherwise, it is not permissible to droop pipe $j$  to the empty box $(w^{-1}(i_3),i_2)$. This implies that there exists a  pipe $j'$  such that $j<j'<i_2$ and $ w^{-1}(j)<w^{-1}(j')<w^{-1}(i_3)$. We first droop   pipe $j'$ to $(w^{-1}(i_3),i_2)$ and then droop  pipe  $j$ to $(w^{-1}(j'),j')$. Continue this process, we can turn $(w^{-1}(j),j)$ into an empty box eventually.  
\end{proof}

By  \textbf{Claim 1}, suppose that $(p_1,q),(p_2,q)$ are the leftmost and topmost two empty boxes located southeast of pipe $i_1$. We obtain $\mathcal{P}_1$ and $\mathcal{P}_2$ by drooping pipe $i_1$  to $(p_1,q)$ and $(p_2,q)$ respectively, see Figure \ref{fig:1432 case droop respectively}. 

\begin{figure}[ht]
	\centering
	\begin{minipage}{0.18\textwidth}
		\begin{tikzpicture}[scale=0.45]
		\draw[very thin, step=1] (0,0) rectangle (4,6); 
		\draw[fill=lightgray](2,1) rectangle (3,2);
		\draw[fill=lightgray](2,3) rectangle (3,4);
		\draw[thick, blue](0.5,0)--(0.5,5.5)--(4,5.5);
        \draw[draw=red,dashed](0,6) rectangle (3,3);
		\node[below] at (0.5,0) {$i_1$};
		\node[below] at (2.5,0) {$q$};
		\node[right] at (4,5.5) {\footnotesize$w^{-1}(i_1)$};
		\node[right] at (4,3.5) {$p_1$};
		\node[right] at (4,1.5) {$p_2$};
		\end{tikzpicture}
	\end{minipage}
$\rightarrow$
 	\begin{minipage}{0.18\textwidth}
		\begin{tikzpicture}[scale=0.45]
		\draw[very thin, step=1] (0,0) rectangle (4,6); 
		\draw[fill=lightgray](2,1) rectangle (3,2);
		\draw[fill=lightgray](0,5) rectangle (1,6);
		\draw[thick, blue](0.5,0)--(0.5,3.5)--(2.5,3.5)--(2.5,5.5)--(4,5.5);
		\node[below] at (0.5,0) {$i_1$};
		\node[below] at (2.5,0) {$q$};
		\node[right] at (4,5.5) {\footnotesize$w^{-1}(i_1)$};
		\node[right] at (4,3.5) {$p_1$};
		\node[right] at (4,1.5) {$p_2$};
		\end{tikzpicture}
	\end{minipage}
 \qquad
 \qquad
	\begin{minipage}{0.18\textwidth}
		\begin{tikzpicture}[scale=0.45]
		\draw[very thin, step=1] (0,0) rectangle (4,6); 
		\draw[fill=lightgray](2,1) rectangle (3,2);
		\draw[fill=lightgray](2,3) rectangle (3,4);
		\draw[thick, blue](0.5,0)--(0.5,5.5)--(4,5.5);
        \draw[draw=red,dashed](0,6) rectangle (3,1);
		\node[below] at (0.5,0) {$i_1$};
		\node[below] at (2.5,0) {$q$};
		\node[right] at (4,5.5) {\footnotesize$w^{-1}(i_1)$};
		\node[right] at (4,3.5) {$p_1$};
		\node[right] at (4,1.5) {$p_2$};
		\end{tikzpicture}
	\end{minipage}
$\rightarrow$
  	\begin{minipage}{0.18\textwidth}
		\begin{tikzpicture}[scale=0.45]
		\draw[very thin, step=1] (0,0) rectangle (4,6); 
		\draw[fill=lightgray](0,5) rectangle (1,6);
  	\draw[fill=lightgray](0,3) rectangle (1,4);
		\draw[thick, blue](0.5,0)--(0.5,1.5)--(2.5,1.5)--(2.5,5.5)--(4,5.5);
		\node[below] at (0.5,0) {$i_1$};
		\node[below] at (2.5,0) {$q$};
		\node[right] at (4,5.5) {\footnotesize$w^{-1}(i_1)$};
		\node[right] at (4,3.5) {$p_1$};
		\node[right] at (4,1.5) {$p_2$};
		\end{tikzpicture}
	\end{minipage}
	\caption{The construction of $\mathcal{P}_1$ and $\mathcal{P}_2$ in Case 1}
	\label{fig:1432 case droop respectively}
\end{figure}
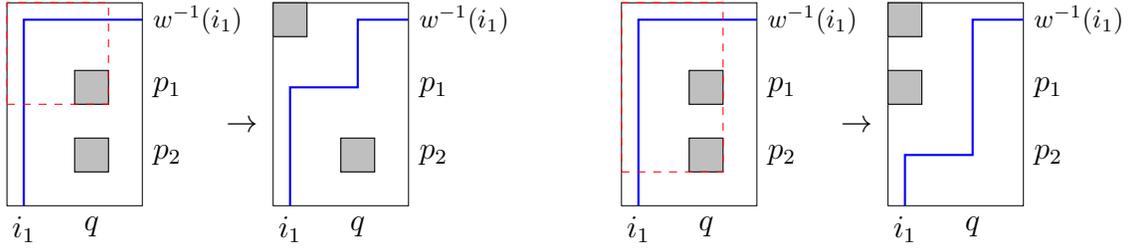

It is easy to see that the empty boxes and NW elbows of $\mathcal{P}_1$ and $\mathcal{P}_2$ only differ in row $p_1$ and row $p_2$. Denote $F(x)$ the polynomial generated by the other rows. Then $\mathcal{P}_1$ and $\mathcal{P}_2$
generate polynomials $(1+x_{p_1})x_{p_2}F(x)$ and $(1+x_{p_2})x_{p_1}F(x),$ respectively.
Thus the polynomial $x_{p_1}x_{p_2}F(x)$ appears at least twice in $\widetilde{\G}_{w}(x)$. 

Notice that the proof for the case when $w$ contains 1432  depends only on the two empty boxes in column $i_2$, thus the proof for the $1342$ case is exactly the same. 

\vspace{.2cm}
\noindent\textbf{Case 2.}   $w$ contains the pattern $31524$. Then there exist $i_1<i_2<i_3<i_4<i_5$ such that $w^{-1}(i_3)<w^{-1}(i_1)<w^{-1}(i_5)<w^{-1}(i_2)<w^{-1}(i_4)$, see      Figure \ref{fig:pattern 31524} for an illustration. Clearly, $(w^{-1}(i_5),i_2)$, $(w^{-1}(i_5),i_4)\in D(w)$. 

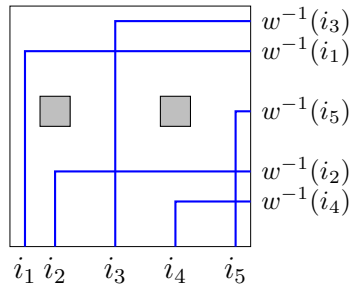
\begin{figure}[ht]
	\centering
	\begin{tikzpicture}[scale=0.4]
	\draw[very thin, step=1] (0,0) rectangle (8,8); 
	\draw[thick, blue](0.5,0)--(0.5,6.5)--(8,6.5);
	\draw[thick, blue](1.5,0)--(1.5,2.5)--(8,2.5);
	\draw[thick, blue](3.5,0)--(3.5,7.5)--(8,7.5);
	\draw[thick, blue](5.5,0)--(5.5,1.5)--(8,1.5);
	\draw[thick, blue](7.5,0)--(7.5,4.5)--(8,4.5);
	\draw[fill=lightgray](1,4) rectangle (2,5);
	\draw[fill=lightgray](5,4) rectangle (6,5);
	\node[below] at (0.5,0) {$i_1$};
	\node[below] at (1.5,0) {$i_2$};
	\node[below] at (3.5,0) {$i_3$};
	\node[below] at (5.5,0) {$i_4$};
	\node[below] at (7.5,0) {$i_5$};	
	\node[right] at (8,7.5) {\footnotesize$w^{-1}(i_3)$};
	\node[right] at (8,6.5) {\footnotesize$w^{-1}(i_1)$};
	\node[right] at (8,4.5) {\footnotesize$w^{-1}(i_5)$};
	\node[right] at (8,2.5) {\footnotesize$w^{-1}(i_2)$};
	\node[right] at (8,1.5) {\footnotesize$w^{-1}(i_4)$};
	\end{tikzpicture} 	
	\caption{The pattern $31524$}
	\label{fig:pattern 31524}
\end{figure}

If we can droop pipe $i_1$ to $(w^{-1}(i_5),i_2)$ and    pipe $i_3$ to $(w^{-1}(i_5),i_4)$ simultaneously, then  the obtained bumpless pipe dream $\mathcal{P}$   generates a polynomial of the form $$F(x)(1+x_{w^{-1}(i_5)})^2,$$ which has a monomial with coefficient larger than $1$. So it suffices to  show that the droops for pipe $i_1$  to $(w^{-1}(i_5),i_2)$ and    pipe $i_3$ to $(w^{-1}(i_5),i_4)$  are both permissible.

\textcircled{1} If there is an SE elbow labeled as $i_1'$ in the region $[w^{-1}(i_1),w^{-1}(i_5)]\times[i_1,i_2]$, view  pipe $i_1'$ as  pipe $i_1$. 

\textcircled{2} If there is an SE elbow labeled as $i_3'$ in the region $[w^{-1}(i_3),w^{-1}(i_1)]\times[i_3,i_4]$, view  pipe $i_3'$  as  pipe $i_3$. 

\textcircled{3} Suppose that there is an SE elbow labeled as $j$ in the region $[w^{-1}(i_1),w^{-1}(i_5)]\times[i_3,i_4]$, as is shown in Figure \ref{fig:Droop not permissible in pattern 31524}. Then we have
$w^{-1}(i_1)<w^{-1}(j)<w^{-1}(i_5)<w^{-1}(i_2)$, which implies that $w$ contains the pattern $1342$ and the proof  goes back to \textbf{Case 1}.

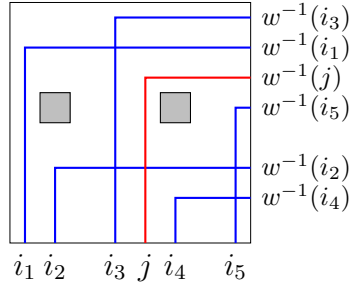
\begin{figure}[ht]
	\centering
	\begin{tikzpicture}[scale=0.4]
	\draw[very thin, step=1] (0,0) rectangle (8,8); 
	\draw[thick, blue](0.5,0)--(0.5,6.5)--(8,6.5);
	\draw[thick, blue](1.5,0)--(1.5,2.5)--(8,2.5);
	\draw[thick, blue](3.5,0)--(3.5,7.5)--(8,7.5);
	\draw[thick, blue](5.5,0)--(5.5,1.5)--(8,1.5);
	\draw[thick, blue](7.5,0)--(7.5,4.5)--(8,4.5);
	\draw[thick, red](4.5,0)--(4.5,5.5)--(8,5.5);
	\draw[fill=lightgray](1,4) rectangle (2,5);
	\draw[fill=lightgray](5,4) rectangle (6,5);
	\node[below] at (0.5,0) {$i_1$};
	\node[below] at (1.5,0) {$i_2$};
	\node[below] at (3.5,0) {$i_3$};
	\node[below] at (5.5,0) {$i_4$};
	\node[below] at (7.5,0) {$i_5$};
	\node[below] at (4.5,0) {$j$};	
	\node[right] at (8,7.5) {\footnotesize$w^{-1}(i_3)$};
	\node[right] at (8,6.5) {\footnotesize$w^{-1}(i_1)$};
	\node[right] at (8,4.5) {\footnotesize$w^{-1}(i_5)$};
	\node[right] at (8,2.5) {\footnotesize$w^{-1}(i_2)$};
	\node[right] at (8,1.5) {\footnotesize$w^{-1}(i_4)$};
	\node[right] at (8,5.5) {\footnotesize$w^{-1}(j)$};
	\end{tikzpicture} 	
	\caption{An impermissible droop in the pattern $31524$}
	\label{fig:Droop not permissible in pattern 31524}
\end{figure}

\noindent\textbf{Case 3.} $w$ contains the pattern $12534$. Then there exist $i_1<i_2<i_3<i_4<i_5$ such that $w^{-1}(i_1)<w^{-1}(i_2)<w^{-1}(i_5)<w^{-1}(i_3)<w^{-1}(i_4)$, and $(w^{-1}(i_5),i_3)$, $(w^{-1}(i_5),i_4)\in D(w)$, see Figure \ref{fig:pattern 12534} for an illustration. 

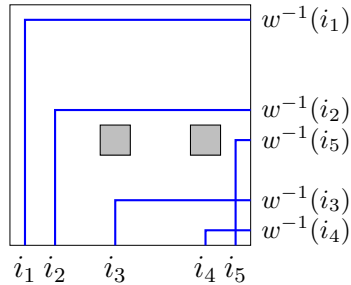
\begin{figure}[ht]
	\centering
	\begin{tikzpicture}[scale=0.4]
	\draw[very thin, step=1] (0,0) rectangle (8,8); 
	\draw[thick, blue](0.5,0)--(0.5,7.5)--(8,7.5);
	\draw[thick, blue](1.5,0)--(1.5,4.5)--(8,4.5);
	\draw[thick, blue](3.5,0)--(3.5,1.5)--(8,1.5);
	\draw[thick, blue](6.5,0)--(6.5,0.5)--(8,0.5);
	\draw[thick, blue](7.5,0)--(7.5,3.5)--(8,3.5);
	\draw[fill=lightgray](3,3) rectangle (4,4);
	\draw[fill=lightgray](6,3) rectangle (7,4);
	\node[below] at (0.5,0) {$i_1$};
	\node[below] at (1.5,0) {$i_2$};
	\node[below] at (3.5,0) {$i_3$};
	\node[below] at (6.5,0) {$i_4$};
	\node[below] at (7.5,0) {$i_5$};	
	\node[right] at (8,7.5) {\footnotesize$w^{-1}(i_1)$};
	\node[right] at (8,4.5) {\footnotesize$w^{-1}(i_2)$};
	\node[right] at (8,3.5) {\footnotesize$w^{-1}(i_5)$};
	\node[right] at (8,1.5) {\footnotesize$w^{-1}(i_3)$};
	\node[right] at (8,0.5) {\footnotesize$w^{-1}(i_4)$};
	\end{tikzpicture} 	
	\caption{The pattern $12534$}
	\label{fig:pattern 12534}
\end{figure}

\vspace{.2cm}
\noindent\textbf{Claim 2:} If $w$ contains the pattern $12534$ and avoids the patterns $1342$, $1432$ and $31524$, then  $(w^{-1}(i_5),j)\in D(w)$ for all $i_3\leq j\leq i_4$. 

\begin{proof}[Proof of Claim 2]
Clearly, $(w^{-1}(i_5),i_3), (w^{-1}(i_5),i_4)\in D(w)$. Suppose $(w^{-1}(i_5),j)\notin D(w)$   for some $i_3<j<i_4$. Then $(w^{-1}(i_5),j)$ must be a vertical tile, which implies  pipe $j$ satisfies $w^{-1}(j)>w^{-1}(i_5)$. The right end of pipe  $j$ has two situations.
	
\textcircled{1}  $w^{-1}(j)<w^{-1}(i_1)$, see Figure \ref{fig:empty box not continue case 1}. Then we see that $w^{-1}(j)<w^{-1}(i_1)<w^{-1}(i_5)<w^{-1}(i_3)<w^{-1}(i_4)$, which implies $w$ contains the pattern $31524$, and proof follows by \textbf{Case 2}. 
	
	\textcircled{2}   $w^{-1}(i_1)<w^{-1}(j)<w^{-1}(i_5)$, see Figure \ref{fig:empty box not continue case 2}.  Then we find $w^{-1}(i_1)<w^{-1}(j)<w^{-1}(i_5)<w^{-1}(i_3)$, which implies $w$ contains the pattern $1342$, and proof follows by \textbf{Case 1}. 
	
Therefore, all boxes between $(w^{-1}(i_5),i_3)$ and $(w^{-1}(i_5),i_4)$ are empty boxes.  
\end{proof}

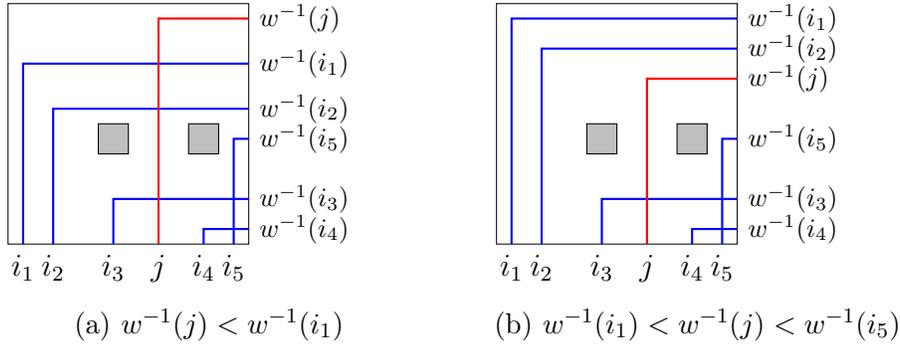
\begin{figure}[ht]
	\centering
	\begin{minipage}{0.35\textwidth}
		\begin{tikzpicture}[scale=0.4]
		\draw[very thin, step=1] (0,0) rectangle (8,8); 
		\draw[thick, blue](0.5,0)--(0.5,6)--(8,6);
		\draw[thick, blue](1.5,0)--(1.5,4.5)--(8,4.5);
		\draw[thick, blue](3.5,0)--(3.5,1.5)--(8,1.5);
		\draw[thick, blue](6.5,0)--(6.5,0.5)--(8,0.5);
		\draw[thick, blue](7.5,0)--(7.5,3.5)--(8,3.5);
		\draw[thick, red](5,0)--(5,7.5)--(8,7.5);
		\draw[fill=lightgray](3,3) rectangle (4,4);
		\draw[fill=lightgray](6,3) rectangle (7,4);
		\node[below] at (0.5,0) {$i_1$};
		\node[below] at (1.5,0) {$i_2$};
		\node[below] at (3.5,0) {$i_3$};
		\node[below] at (6.5,0) {$i_4$};
		\node[below] at (7.5,0) {$i_5$};
		\node[below] at (5,0) {$j$};	
		\node[right] at (8,7.5) {\footnotesize$w^{-1}(j)$};
		\node[right] at (8,4.5) {\footnotesize$w^{-1}(i_2)$};
		\node[right] at (8,3.5) {\footnotesize$w^{-1}(i_5)$};
		\node[right] at (8,1.5) {\footnotesize$w^{-1}(i_3)$};
		\node[right] at (8,0.5) {\footnotesize$w^{-1}(i_4)$};
		\node[right] at (8,6) {\footnotesize$w^{-1}(i_1)$};
		\end{tikzpicture} 
		\subcaption{$w^{-1}(j)<w^{-1}(i_1)$} 
		\label{fig:empty box not continue case 1}
	\end{minipage}
	\qquad
	\begin{minipage}{0.35\textwidth}
		\begin{tikzpicture}[scale=0.4]
		\draw[very thin, step=1] (0,0) rectangle (8,8); 
		\draw[thick, blue](0.5,0)--(0.5,7.5)--(8,7.5);
		\draw[thick, blue](1.5,0)--(1.5,6.5)--(8,6.5);
		\draw[thick, blue](3.5,0)--(3.5,1.5)--(8,1.5);
		\draw[thick, blue](6.5,0)--(6.5,0.5)--(8,0.5);
		\draw[thick, blue](7.5,0)--(7.5,3.5)--(8,3.5);
		\draw[thick, red](5,0)--(5,5.5)--(8,5.5);
		\draw[fill=lightgray](3,3) rectangle (4,4);
		\draw[fill=lightgray](6,3) rectangle (7,4);
		\node[below] at (0.5,0) {$i_1$};
		\node[below] at (1.5,0) {$i_2$};
		\node[below] at (3.5,0) {$i_3$};
		\node[below] at (6.5,0) {$i_4$};
		\node[below] at (7.5,0) {$i_5$};
		\node[below] at (5,0) {$j$};	
		\node[right] at (8,7.5) {\footnotesize$w^{-1}(i_1)$};
		\node[right] at (8,6.5) {\footnotesize$w^{-1}(i_2)$};
		\node[right] at (8,3.5) {\footnotesize$w^{-1}(i_5)$};
		\node[right] at (8,1.5) {\footnotesize$w^{-1}(i_3)$};
		\node[right] at (8,0.5) {\footnotesize$w^{-1}(i_4)$};
		\node[right] at (8,5.5) {\footnotesize$w^{-1}(j)$};
		\end{tikzpicture} 
		\subcaption{$w^{-1}(i_1)<w^{-1}(j)<w^{-1}(i_5)$}
		\label{fig:empty box not continue case 2}
	\end{minipage}
	\caption{Non-consecutive empty boxes in Case 3}
	\label{fig:empty box not continue}
\end{figure}

By \textbf{Claim 2}, both  $(w^{-1}(i_5),i_3)$ and $(w^{-1}(i_5),i_3+1)$ are  empty boxes when $w$ contains the pattern $12534$ and avoids the patterns $1342$, $1432$ and $31524$. 
In the following, we first describe how to construct   $\mathcal{P}_1$ and $\mathcal{P}_2$ from $D(w)$ and then show such droop operations are indeed permissible.

To construct $\mathcal{P}_1$, first droop pipe $i_2$ to the empty box at $(w^{-1}(i_5),i_3)$, thus turning $(w^{-1}(i_2),i_2)$ into a new empty box. Then droop pipe $i_1$ to the empty box at $(w^{-1}(i_2),i_2)$, see Figure \ref{fig:The construction of P_1 in Case3}.

\begin{figure}[ht]
	\centering
	\begin{minipage}{0.3\textwidth}
    \begin{tikzpicture}[scale=0.4]
        \draw[very thin] (0,0) rectangle (8,8); 
		\draw[fill=lightgray](6,3) rectangle (7,4);
  	\draw[fill=lightgray](5,3) rectangle (6,4);
		\draw[thick, blue](0.5,0)--(0.5,7.5)--(8,7.5);
		\draw[thick, blue](2.5,0)--(2.5,5.5)--(8,5.5);
        \draw[draw=red,dashed](2,6) rectangle (6,3);
		\node[below] at (0.5,0) {$i_1$};
		\node[below] at (2.5,0) {$i_2$};
		\node[below] at (5.5,0) {$i_3$};
		\node[below right] at (5.8,0) {$i_3$+$1$};	
		\node[right] at (8,7.5) {\footnotesize$w^{-1}(i_1)$};
		\node[right] at (8,5.5) {\footnotesize$w^{-1}(i_2)$};
		\node[right] at (8,3.5) {\footnotesize$w^{-1}(i_5)$};
    \end{tikzpicture}      
	\end{minipage}
	\quad
	\begin{minipage}{0.3\textwidth}
    \begin{tikzpicture}[scale=0.4]
        \draw[very thin] (0,0) rectangle (8,8); 
		\draw[fill=lightgray](6,3) rectangle (7,4);
		\draw[fill=lightgray](2,5) rectangle (3,6);
		\draw[thick, blue](0.5,0)--(0.5,7.5)--(8,7.5);
		\draw[thick, blue](2.5,0)--(2.5,3.5)--(5.5,3.5)--(5.5,5.5)--(8,5.5);
        \draw[draw=red,dashed](0,8) rectangle (3,5);
		\node[below] at (0.5,0) {$i_1$};
		\node[below] at (2.5,0) {$i_2$};
		\node[below] at (5.5,0) {$i_3$};
		\node[below right] at (5.8,0) {$i_3$+$1$};	
		\node[right] at (8,7.5) {\footnotesize$w^{-1}(i_1)$};
		\node[right] at (8,5.5) {\footnotesize$w^{-1}(i_2)$};
		\node[right] at (8,3.5) {\footnotesize$w^{-1}(i_5)$};
    \end{tikzpicture}
    \end{minipage}
    \quad
    \begin{minipage}{0.3\textwidth}
        \begin{tikzpicture}[scale=0.4]
		\draw[very thin] (0,0) rectangle (8,8);
		\draw[fill=lightgray](6,3) rectangle (7,4);
        \draw[fill=lightgray](0,7) rectangle (1,8);
		\draw[thick, blue](0.5,0)--(0.5,5.5)--(2.5,5.5)--(2.5,7.5)--(8,7.5);
		\draw[thick, blue](2.5,0)--(2.5,3.5)--(5.5,3.5)--(5.5,5.5)--(8,5.5);
		\node[below] at (0.5,0) {$i_1$};
		\node[below] at (2.5,0) {$i_2$};
		\node[below] at (5.5,0) {$i_3$};
		\node[below right] at (5.8,0) {$i_3$+$1$};	
		\node[right] at (8,7.5) {\footnotesize$w^{-1}(i_1)$};
		\node[right] at (8,5.5) {\footnotesize$w^{-1}(i_2)$};
		\node[right] at (8,3.5) {\footnotesize$w^{-1}(i_5)$};
		\end{tikzpicture}
    \end{minipage}
	\caption{The construction of $\mathcal{P}_1$ in Case 3}
 \label{fig:The construction of P_1 in Case3}
\end{figure}
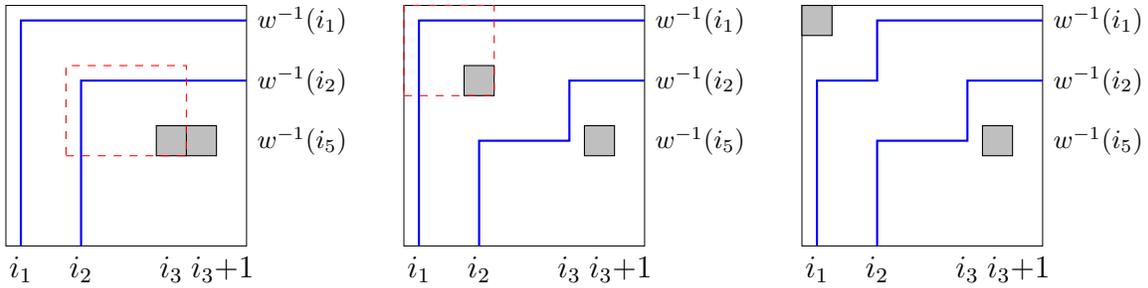

To construct $\mathcal{P}_2$, first droop pipe $i_2$ to the empty box at $(w^{-1}(i_5),i_3+1)$. Then droop pipe $i_1$ to the empty box at $(w^{-1}(i_2),i_2)$, see Figure \ref{fig:The construction of P_2 in Case 3}. 

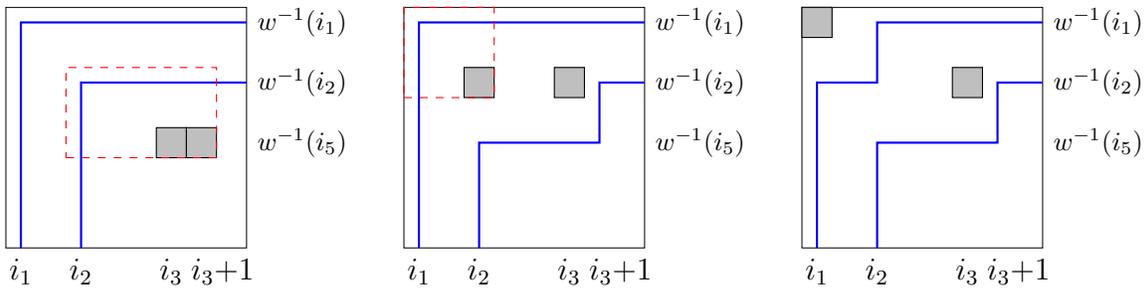
\begin{figure}[ht]
    \centering
    \begin{minipage}{0.3\textwidth}
    \begin{tikzpicture}[scale=0.4]
        \draw[very thin] (0,0) rectangle (8,8); 
		\draw[fill=lightgray](5,3) rectangle (6,4);
		\draw[fill=lightgray](6,3) rectangle (7,4);
		\draw[thick, blue](0.5,0)--(0.5,7.5)--(8,7.5);
		\draw[thick, blue](2.5,0)--(2.5,5.5)--(8,5.5);
        \draw[draw=red,dashed](2,6) rectangle (7,3);
		\node[below] at (0.5,0) {$i_1$};
		\node[below] at (2.5,0) {$i_2$};
		\node[below] at (5.5,0) {$i_3$};
		\node[below right] at (5.8,0) {$i_3$+$1$};	
		\node[right] at (8,7.5) {\footnotesize$w^{-1}(i_1)$};
		\node[right] at (8,5.5) {\footnotesize$w^{-1}(i_2)$};
		\node[right] at (8,3.5) {\footnotesize$w^{-1}(i_5)$};
    \end{tikzpicture}      
	\end{minipage}
	\quad
 \begin{minipage}{0.3\textwidth}
     \begin{tikzpicture}[scale=0.4]
        \draw[very thin] (0,0) rectangle (8,8); 
		\draw[fill=lightgray](2,5) rectangle (3,6);
		\draw[fill=lightgray](5,5) rectangle (6,6);
		\draw[thick, blue](0.5,0)--(0.5,7.5)--(8,7.5);
        \draw[thick, blue](2.5,0)--(2.5,3.5)--(6.5,3.5)--(6.5,5.5)--(8,5.5);
        \draw[draw=red,dashed](0,8) rectangle (3,5);
		\node[below] at (0.5,0) {$i_1$};
		\node[below] at (2.5,0) {$i_2$};
		\node[below] at (5.5,0) {$i_3$};
		\node[below right] at (5.8,0) {$i_3$+$1$};	
		\node[right] at (8,7.5) {\footnotesize$w^{-1}(i_1)$};
		\node[right] at (8,5.5) {\footnotesize$w^{-1}(i_2)$};
		\node[right] at (8,3.5) {\footnotesize$w^{-1}(i_5)$};
    \end{tikzpicture}      
	\end{minipage}
	\quad
    \begin{minipage}{0.3\textwidth}
        \begin{tikzpicture}[scale=0.4]
		\draw[very thin, step=1] (0,0) rectangle (8,8); 
		\draw[fill=lightgray](0,7) rectangle (1,8);
		\draw[fill=lightgray](5,5) rectangle (6,6);
		\draw[thick, blue](0.5,0)--(0.5,5.5)--(2.5,5.5)--(2.5,7.5)--(8,7.5);
		\draw[thick, blue](2.5,0)--(2.5,3.5)--(6.5,3.5)--(6.5,5.5)--(8,5.5);
		\node[below] at (0.5,0) {$i_1$};
		\node[below] at (2.5,0) {$i_2$};
		\node[below] at (5.5,0) {$i_3$};
		\node[below right] at (5.8,0) {$i_3$+$1$};	
		\node[right] at (8,7.5) {\footnotesize$w^{-1}(i_1)$};
		\node[right] at (8,5.5) {\footnotesize$w^{-1}(i_2)$};
		\node[right] at (8,3.5) {\footnotesize$w^{-1}(i_5)$};
		\end{tikzpicture}
    \end{minipage}
    \caption{The construction of $\mathcal{P}_2$ in Case 3}
    \label{fig:The construction of P_2 in Case 3}
\end{figure}

Now we explain the droops in constructing  $\mathcal{P}_1$ and  $\mathcal{P}_2$ are all permissible. For  $\mathcal{P}_1$, if the first droop is not permissible, then there is an SE elbow in the region $[w^{-1}(i_2),w^{-1}(i_5)]\times[i_2,i_3]$, view the new pipe as pipe $i_2$. Similarly, if the second step is not permissible, then there is an SE elbow in the region $[w^{-1}(i_1),w^{-1}(i_2)]\times[i_1,i_2]$,  view the new pipe as pipe $i_1$.  The case for the construction of  $\mathcal{P}_2$ is exactly the same.

 The empty boxes and NW elbows of $\mathcal{P}_1$ and $\mathcal{P}_2$ only differs in row $w^{-1}(i_2)$ and row $w^{-1}(i_5)$. We still use $F(x)$ to denote the polynomial generated by the other rows. For simplicity, let $p_1=w^{-1}(i_2)$ and $p_2=w^{-1}(i_5)$. Thus $\mathcal{P}_1$ and $\mathcal{P}_2$ generate the following two polynomials:
$$F(x)(1+x_{p_1})x_{p_1}^{k_1}(1+x_{p_2})x_{p_2}^{k_2+1},$$
$$F(x)(1+x_{p_1})x_{p_1}^{k_1+1}(1+x_{p_2})x_{p_2}^{k_2},$$
where $k_1$ and $k_2$ are the number of empty boxes in row $w^{-1}(i_2)$ and row $w^{-1}(i_5)$, respectively. Expand the two polynomials above we find $F(x)x_{p_1}^{k_1+1}x_{p_2}^{k_2+1}$ appears at least twice in $\widetilde{\G}_{w}(x)$. 

\vspace{.2cm}
\noindent\textbf{Case 4.} $w$ contains the pattern $21534$. Then there exist $i_1<i_2<i_3<i_4<i_5$ such that $w^{-1}(i_2)<w^{-1}(i_1)<w^{-1}(i_5)<w^{-1}(i_3)<w^{-1}(i_4)$, and $(w^{-1}(i_5),i_3)$, $(w^{-1}(i_5),i_4)\in D(w)$,  see Figure \ref{fig:pattern 21534} for an illustration. 

\begin{figure}[ht]
	\centering
	\begin{tikzpicture}[scale=0.4]
	\draw[very thin, step=1] (0,0) rectangle (8,8); 
	\draw[thick, blue](0.5,0)--(0.5,5.5)--(8,5.5);
	\draw[thick, blue](1.5,0)--(1.5,7.5)--(8,7.5);
	\draw[thick, blue](3.5,0)--(3.5,1.5)--(8,1.5);
	\draw[thick, blue](6.5,0)--(6.5,0.5)--(8,0.5);
	\draw[thick, blue](7.5,0)--(7.5,3.5)--(8,3.5);
	\draw[fill=lightgray](3,3) rectangle (4,4);
	\draw[fill=lightgray](6,3) rectangle (7,4);
	\node[below] at (0.5,0) {$i_1$};
	\node[below] at (1.5,0) {$i_2$};
	\node[below] at (3.5,0) {$i_3$};
	\node[below] at (6.5,0) {$i_4$};
	\node[below] at (7.5,0) {$i_5$};	
	\node[right] at (8,7.5) {\footnotesize$w^{-1}(i_2)$};
	\node[right] at (8,5.5) {\footnotesize$w^{-1}(i_1)$};
	\node[right] at (8,3.5) {\footnotesize$w^{-1}(i_5)$};
	\node[right] at (8,1.5) {\footnotesize$w^{-1}(i_3)$};
	\node[right] at (8,0.5) {\footnotesize$w^{-1}(i_4)$};
	\end{tikzpicture} 	
	\caption{The pattern $21534$}
	\label{fig:pattern 21534}
\end{figure}
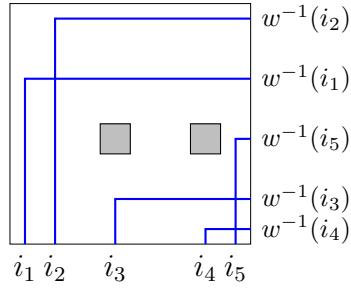

One can easily check that  \textbf{Claim 2} still holds when $w$ contains the pattern $21534$ and avoids the patterns $1342$, $1432$ and $31524$. Thus  all boxes between $(w^{-1}(i_5),i_3)$ and $(w^{-1}(i_5),i_4)$ in row $w^{-1}(i_5)$ are still empty boxes. Again, we will use the two empty boxes at $(w^{-1}(i_5),i_3)$ and $(w^{-1}(i_5),i_3+1)$ to construct $\mathcal{P}_1$ and $\mathcal{P}_2$.

To construct $\mathcal{P}_1$, first droop   pipe $i_1$ to the empty box at $(w^{-1}(i_5),i_3)$. Then apply a $K$-theoretic  droop to pipe $i_1$ and pipe $i_2$ in the dashed area, see Figure \ref{fig:21534 construct P_1}.  

\begin{figure}[ht]
	\centering
	\begin{minipage}{0.3\textwidth}
		\begin{tikzpicture}[scale=0.4]
		\draw[very thin] (0,0) rectangle (8,8); 
		\draw[fill=lightgray](6,3) rectangle (7,4);
  	\draw[fill=lightgray](5,3) rectangle (6,4);
		\draw[thick, blue](0.5,0)--(0.5,5.5)--(8,5.5);
		\draw[thick, blue](2.5,0)--(2.5,7.5)--(8,7.5);
        \draw[draw=red,dashed](0,6) rectangle (6,3);
		\node[below] at (0.5,0) {$i_1$};
		\node[below] at (2.5,0) {$i_2$};
		\node[below] at (5.5,0) {$i_3$};
		\node[below right] at (5.8,0) {$i_3$+$1$};		
		\node[right] at (8,7.5) {\footnotesize$w^{-1}(i_2)$};
		\node[right] at (8,5.5) {\footnotesize$w^{-1}(i_1)$};
		\node[right] at (8,3.5) {\footnotesize$w^{-1}(i_5)$};
		\end{tikzpicture}
	\end{minipage}
	\quad
	\begin{minipage}{0.3\textwidth}
		\begin{tikzpicture}[scale=0.4]
		\draw[very thin] (0,0) rectangle (8,8); 
		\draw[fill=lightgray](0,5) rectangle (1,6);
		\draw[fill=lightgray](6,3) rectangle (7,4);
		\draw[thick, blue](0.5,0)--(0.5,3.5)--(5.5,3.5)--(5.5,5.5)--(8,5.5);
		\draw[thick, blue](2.5,0)--(2.5,7.5)--(8,7.5);
		\draw[dashed,red](2,3) rectangle (6,8);
		\node[below] at (0.5,0) {$i_1$};
		\node[below] at (2.5,0) {$i_2$};
		\node[below] at (5.5,0) {$i_3$};
		\node[below right] at (5.8,0) {$i_3$+$1$};	
		\node[right] at (8,7.5) {\footnotesize$w^{-1}(i_2)$};
		\node[right] at (8,5.5) {\footnotesize$w^{-1}(i_1)$};
		\node[right] at (8,3.5) {\footnotesize$w^{-1}(i_5)$};
		\end{tikzpicture}
	\end{minipage}
	\quad
	\begin{minipage}{0.3\textwidth}
		\begin{tikzpicture}[scale=0.4]
		\draw[very thin] (0,0) rectangle (8,8); 
		\draw[fill=lightgray](0,5) rectangle (1,6);
		\draw[fill=lightgray](6,3) rectangle (7,4);
        \draw[fill=lightgray](2,7) rectangle (3,8);
		\draw[thick, blue](0.5,0)--(0.5,3.5)--(5.5,3.5)--(5.5,7.5)--(8,7.5);
		\draw[thick, blue](2.5,0)--(2.5,5.5)--(8,5.5);
		\draw[dashed,blue](2.5,3.5)--(2.5,5.5)--(6.5,5.5);
		\node[below] at (0.5,0) {$i_1$};
		\node[below] at (2.5,0) {$i_2$};
		\node[below] at (5.5,0) {$i_3$};
		\node[below right] at (5.8,0) {$i_3$+$1$};	
		\node[right] at (8,7.5) {\footnotesize$w^{-1}(i_2)$};
		\node[right] at (8,5.5) {\footnotesize$w^{-1}(i_1)$};
		\node[right] at (8,3.5) {\footnotesize$w^{-1}(i_5)$};
		\end{tikzpicture}
	\end{minipage}
	\caption{The construction of $\mathcal{P}_1$ in Case 4}
	\label{fig:21534 construct P_1}
\end{figure}
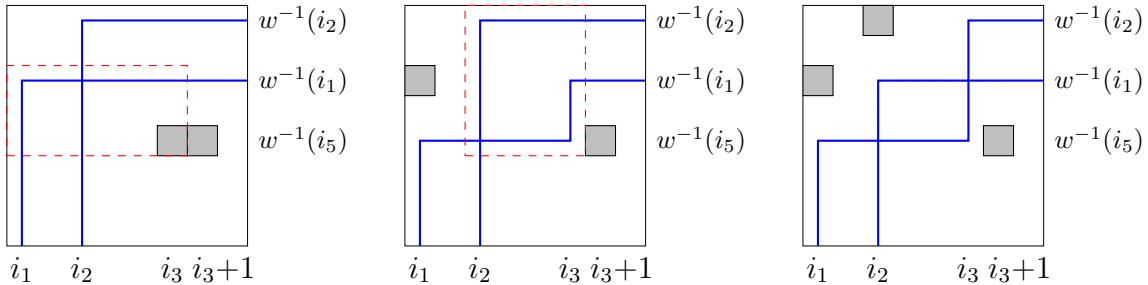

To construct $\mathcal{P}_2$, first droop   pipe $i_1$ to the empty box at $(w^{-1}(i_5),i_3+1)$, thus turning $(w^{-1}(i_1),i_3)$ into a new empty box. Then   droop pipe $i_2$ to the new empty box at $(w^{-1}(i_1),i_3)$, as shown in Figure \ref{fig:21534 construct P_2}.

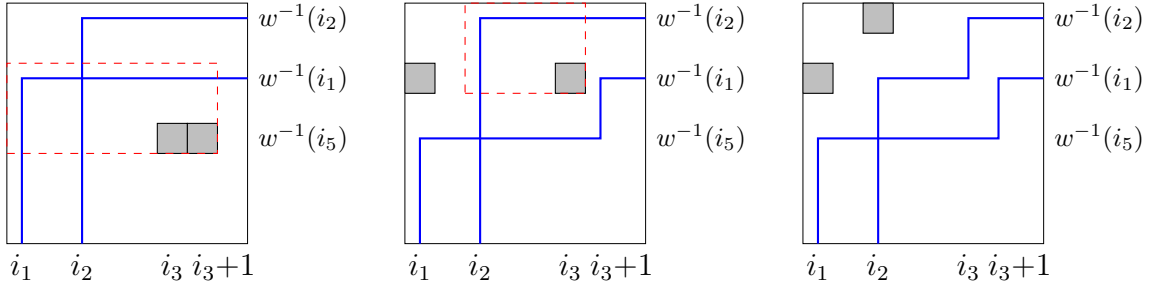
\begin{figure}[ht]
	\centering
	\begin{minipage}{0.3\textwidth}
		\begin{tikzpicture}[scale=0.4]
		\draw[very thin] (0,0) rectangle (8,8); 
		\draw[fill=lightgray](5,3) rectangle (6,4);
    	\draw[fill=lightgray](6,3) rectangle (7,4);
		\draw[thick, blue](0.5,0)--(0.5,5.5)--(8,5.5);
		\draw[thick, blue](2.5,0)--(2.5,7.5)--(8,7.5);
        \draw[draw=red,dashed](0,6) rectangle (7,3);
		\node[below] at (0.5,0) {$i_1$};
		\node[below] at (2.5,0) {$i_2$};
		\node[below] at (5.5,0) {$i_3$};
		\node[below right] at (5.8,0) {$i_3$+$1$};		
		\node[right] at (8,7.5) {\footnotesize$w^{-1}(i_2)$};
		\node[right] at (8,5.5) {\footnotesize$w^{-1}(i_1)$};
		\node[right] at (8,3.5) {\footnotesize$w^{-1}(i_5)$};
		\end{tikzpicture}
	\end{minipage}
	\quad
	\begin{minipage}{0.3\textwidth}
		\begin{tikzpicture}[scale=0.4]
		\draw[very thin] (0,0) rectangle (8,8); 
		\draw[fill=lightgray](5,5) rectangle (6,6);
    	\draw[fill=lightgray](0,5) rectangle (1,6);
		\draw[thick, blue](0.5,0)--(0.5,3.5)--(6.5,3.5)--(6.5,5.5)--(8,5.5);
		\draw[thick, blue](2.5,0)--(2.5,7.5)--(8,7.5);
        \draw[draw=red,dashed](2,8) rectangle (6,5);
		\node[below] at (0.5,0) {$i_1$};
		\node[below] at (2.5,0) {$i_2$};
		\node[below] at (5.5,0) {$i_3$};
		\node[below right] at (5.8,0) {$i_3$+$1$};	
		\node[right] at (8,7.5) {\footnotesize$w^{-1}(i_2)$};
		\node[right] at (8,5.5) {\footnotesize$w^{-1}(i_1)$};
		\node[right] at (8,3.5) {\footnotesize$w^{-1}(i_5)$};
		\end{tikzpicture}
	\end{minipage}
	\quad
	\begin{minipage}{0.3\textwidth}
		\begin{tikzpicture}[scale=0.4]
		\draw[very thin] (0,0) rectangle (8,8); 
		\draw[fill=lightgray](2,7) rectangle (3,8);
    	\draw[fill=lightgray](0,5) rectangle (1,6);
		\draw[thick, blue](0.5,0)--(0.5,3.5)--(6.5,3.5)--(6.5,5.5)--(8,5.5);
		\draw[thick, blue](2.5,0)--(2.5,5.5)--(5.5,5.5)--(5.5,7.5)--(8,7.5);
		\node[below] at (0.5,0) {$i_1$};
		\node[below] at (2.5,0) {$i_2$};
		\node[below] at (5.5,0) {$i_3$};
		\node[below right] at (5.8,0) {$i_3$+$1$};	
		\node[right] at (8,7.5) {\footnotesize$w^{-1}(i_2)$};
		\node[right] at (8,5.5) {\footnotesize$w^{-1}(i_1)$};
		\node[right] at (8,3.5) {\footnotesize$w^{-1}(i_5)$};
		\end{tikzpicture}
	\end{minipage}
	\caption{The construction of $\mathcal{P}_2$ in Case 4}
	\label{fig:21534 construct P_2}
\end{figure}

Again we need to explain all the droops in constructing  $\mathcal{P}_1$ and  $\mathcal{P}_2$ are permissible. Suppose that there is an SE elbow labeled $j$ in the region $[w^{-1}(i_1),w^{-1}(i_5)]\times[i_2,i_3]$, then there is a subsequence $i_2<j<i_3<i_4<i_5$ such that $w^{-1}(i_2)<w^{-1}(j)<w^{-1}(i_5)<w^{-1}(i_3)<w^{-1}(i_4)$, which implies $w$ contains the pattern $12534$ and the proof returns to \textbf{Case 3}. 
If there is an SE elbow is in $[w^{-1}(i_2),w^{-1}(i_1)]\times [i_2,i_3]$ or $[w^{-1}(i_1),w^{-1}(i_5)]\times [i_1,i_2]$, similarly $w$ contains a 12534 pattern.

Let $p_1=w^{-1}(i_1)$ and $p_2=w^{-1}(i_5)$. 
By construction, $\mathcal{P}_1$ and $\mathcal{P}_2$ generate the following two polynomials:
\begin{align*}
&F(x)x_{p_1}^{k_1}x_{p_2}^{k_2+1}(1+x_{p_2}),\\
&F(x)x_{p_1}^{k_1}(1+x_{p_1})x_{p_2}^{k_2}(1+x_{p_2}), 
\end{align*}
where $k_1$ and $k_2$ are the number of empty boxes in row $w^{-1}(i_1)$ and row $w^{-1}(i_5)$, respectively. Expand the above two polynomials  we find that $F(x)x_{p_1}^{k_1}x_{p_2}^{k_2+1}$ appears at least twice in $\widetilde{\G}_{w}(x)$. 
\end{proof}

Next we prove the sufficiency of Theorem \ref{thm:main theorem}.


\begin{thm}\label{thm:sufficiency}
	If $w$ avoids the patterns $1432$, $1342$, $13254$, $31524$, $12534$ and $21534$, then the coefficients of $\widetilde{\G}_{w}(x)$ equal to either $0$ or $1$. 
\end{thm}

To prove Theorem \ref{thm:sufficiency}, we give a factorization of $\widetilde{\G}_{w}(x)$ when $w$ avoids the above six patterns. 
We list the following lemma for convenience of quotation,  which holds by definition.

\begin{lem}\label{lem:Rothe pipe dream and pattern avoidence}
Let $n>k$,	$w\in S_n$ contains the pattern $\sigma\in S_k$ if and only if $D(w)$ contains $D(\sigma)$. More precisely,  there exist $i_1<\cdots<i_k$ such that $w(i_1)\cdots w(i_k)$ has the same relative order with $\sigma$, and $D(\sigma)$ can be obtained from $D(w)$ by deleting the rows $[n]\backslash \{i_1,\ldots,i_k\}$ and columns $[n]\backslash\{w(i_1),\ldots,w(i_k)\}$.
\end{lem}

\begin{prop}\label{prop:description of the Rothe pipe dream when avoiding 6 patterns}
Assume that $w$ avoids the patterns $1432$, $1342$, $13254$, $31524$, $12534$ and $21534$. If there are $k\ (k\ge2)$ empty boxes in the southeast region of any pipe $i\ (1\le i\le n)$ in $D(w)$, then these $k$ empty boxes are all in row $w^{-1}(i)+1$ and consecutively arranged from column $i+1$ to column $i+k$. Moreover, pipe $i$ is the unique pipe to the northwest of these $k$ boxes.
	
\end{prop}

\begin{proof}
Firstly, we show that the $k$ empty boxes in the southeast region of any pipe $i$ must be in the same row.  	
Suppose to the contrary that there exist two empty boxes at $(p_1,q_1),(p_2,q_2)$ with $p_1\neq p_2$. There are three cases depending on the relative positions of $(p_1,q_1)$ and $(p_2,q_2)$, see Figure \ref{fig:Two empty boxes in different rows} for an illustration.

If $q_1=q_2$, then by Lemma \ref{lem:Rothe pipe dream and pattern avoidence}, $D(w)$ contains $D(1342)$. If $p_2<p_1$ and  $q_1<q_2$, then $D(w)$ contains $D(1432)$. If $p_1<p_2$ and  $q_1<q_2$, then $D(w)$ contains $D(13254)$. In conclusion, if there are two empty boxes in different rows, then $w$ contains one of the patterns $1432$, $1342$ and $13254$, a contradiction.

	\begin{figure}[ht]
		\centering
  		\begin{minipage}{0.23\textwidth}
			\begin{tikzpicture}[scale=0.4]
			\draw[very thin] (0,0) rectangle (5,7); 
			\draw[fill=lightgray](2,1) rectangle (3,2);
			\draw[fill=lightgray](2,4) rectangle (3,5);
			\draw[thick, blue](0.5,0)--(0.5,6.5)--(5,6.5);
			\node[below] at (0.5,0) {$i$};
			\node[below] at (2.5,0) {$q_1$};
			\node[right] at (5,6.5) {\footnotesize$w^{-1}(i)$};
			\node[right] at (5,4.5) {$p_2$};
			\node[right] at (5,1.5) {$p_1$};
			\end{tikzpicture}
			\subcaption{$q_1=q_2$}
		\end{minipage}
  \quad
		\begin{minipage}{0.3\textwidth}
			\begin{tikzpicture}[scale=0.4]
			\draw[very thin] (0,0) rectangle (7,7); 
			\draw[fill=lightgray](2,1) rectangle (3,2);
			\draw[fill=lightgray](5,4) rectangle (6,5);
			\draw[fill=lightgray](2,4) rectangle (3,5);
			\draw[thick, blue](0.5,0)--(0.5,6.5)--(7,6.5);
			\node[below] at (0.5,0) {$i$};
			\node[below] at (2.5,0) {$q_1$};
			\node[below] at (5.5,0) {$q_2$};	
			\node[right] at (7,6.5) {\footnotesize$w^{-1}(i)$};
			\node[right] at (7,4.5) {$p_2$};
			\node[right] at (7,1.5) {$p_1$};
			\end{tikzpicture}
			\subcaption{$p_2< p_1,q_1<q_2$}
		\end{minipage}
	\quad
		\begin{minipage}{0.3\textwidth}
			\begin{tikzpicture}[scale=0.4]
			\draw[very thin] (0,0) rectangle (7,7);
			\draw[fill=lightgray](5,1) rectangle (6,2);
			\draw[fill=lightgray](2,4) rectangle (3,5);
			\draw[thick, blue](0.5,0)--(0.5,6.5)--(7,6.5);
			\node[below] at (0.5,0) {$i$};
			\node[below] at (2.5,0) {$q_1$};
			\node[below] at (5.5,0) {$q_2$};	
			\node[right] at (7,6.5) {\footnotesize$w^{-1}(i)$};
			\node[right] at (7,4.5) {$p_1$};
			\node[right] at (7,1.5) {$p_2$};
			\end{tikzpicture}
			\subcaption{$p_1< p_2,q_1<q_2$}
		\end{minipage}
		\caption{Two empty boxes in different rows}
		\label{fig:Two empty boxes in different rows}
	\end{figure}
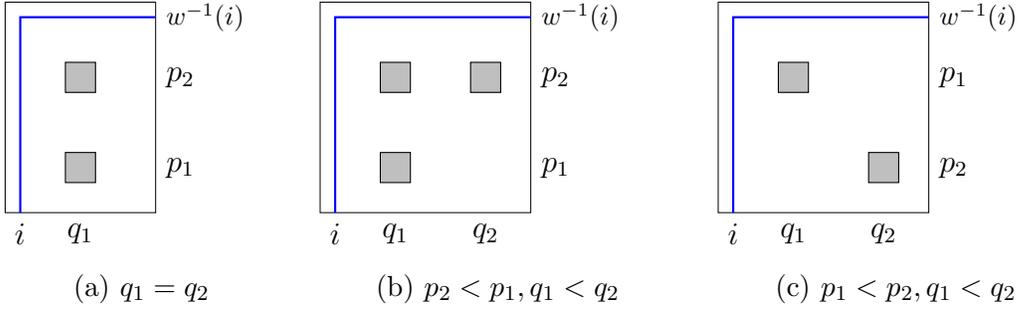
	
Secondly, we show that the $k$ empty boxes in the southeast region of pipe $i$ are all in row $w^{-1}(i)+1$ and the leftmost box is in column $i+1$.  
Since all the $k$ boxes are in the same row, if they are not in row $w^{-1}(i)+1$, then the row $w^{-1}(i)+1$ has no empty boxes to the southeast region of pipe $i$. That is, for all $j>i$, $(w^{-1}(i)+1,j)\notin D(w)$. Then there exists a pipe $i'$ with $w^{-1}(i')=w^{-1}(i)+1$ and the $k$ empty boxes are to the southeast of pipe $i'$, as shown in Figure \ref{fig:k boxes not in row w^{-1}(i)+1}. If $i'<i$, then by Lemma \ref{lem:Rothe pipe dream and pattern avoidence}, $D(w)$ contains $D(21534)$. If $i'>i$, then $D(w)$ contains $D(12534)$.

If the leftmost empty box is not in column $i+1$, then the $k$ empty boxes are to the southeast of the pipe $i+1$, and $D(w)$   is similar to Figure \ref{fig:21534 with i'<i}, which forces $w$ to contain the pattern  $21534$.

	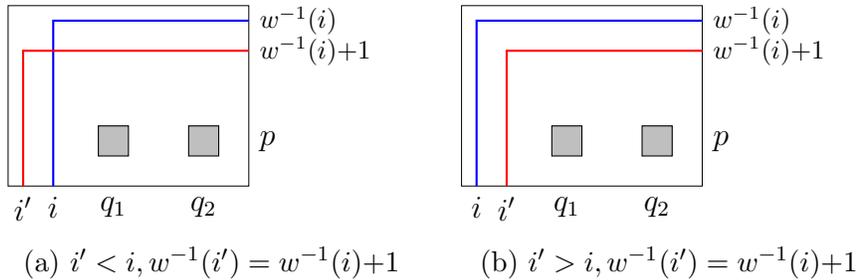
\begin{figure}[ht]
		\centering
		\begin{minipage}{0.35\textwidth}
			\begin{tikzpicture}[scale=0.4]
			\draw[very thin] (0,0) rectangle (8,6); 
			\draw[fill=lightgray](6,1) rectangle (7,2);
			\draw[fill=lightgray](3,1) rectangle (4,2);
			\draw[thick, blue](1.5,0)--(1.5,5.5)--(8,5.5);
			\draw[thick, red](0.5,0)--(0.5,4.5)--(8,4.5);
			\node[below] at (0.5,0) {$i'$};
			\node[below] at (1.5,0) {$i$};
			\node[below] at (3.5,0) {$q_1$};
			\node[below] at (6.5,0) {$q_2$};	
			\node[right] at (8,5.5) {\footnotesize$w^{-1}(i)$};
			\node[right] at (8,4.5) {\footnotesize{$w^{-1}(i)$+$1$}};
			\node[right] at (8,1.5) {$p$};
			\end{tikzpicture}
			\subcaption{$i'<i, w^{-1}(i')=w^{-1}(i)$+$1$}
            \label{fig:21534 with i'<i}
		\end{minipage}
		\quad
		\begin{minipage}{0.35\textwidth}
			\begin{tikzpicture}[scale=0.4]
			\draw[very thin] (0,0) rectangle (8,6); 
			\draw[fill=lightgray](6,1) rectangle (7,2);
			\draw[fill=lightgray](3,1) rectangle (4,2);
			\draw[thick, blue](0.5,0)--(0.5,5.5)--(8,5.5);
			\draw[thick, red](1.5,0)--(1.5,4.5)--(8,4.5);
			\node[below] at (1.5,0) {$i'$};
			\node[below] at (0.5,0) {$i$};
			\node[below] at (3.5,0) {$q_1$};
			\node[below] at (6.5,0) {$q_2$};	
			\node[right] at (8,5.5) {\footnotesize$w^{-1}(i)$};
			\node[right] at (8,4.5) {\footnotesize{$w^{-1}(i)$+$1$}};
			\node[right] at (8,1.5) {$p$};
			\end{tikzpicture}
			\subcaption{$i'>i, w^{-1}(i')=w^{-1}(i)$+$1$}
		\end{minipage}
		\caption{$k$  empty boxes not in row $w^{-1}(i)$+$1$}
		\label{fig:k boxes not in row w^{-1}(i)+1}
	\end{figure}
	
Finally,  we show that the $k$ empty boxes are consecutively arranged from column $i+1$ to column $i+k$. 
Suppose $(w^{-1}(i)+1,q_1),(w^{-1}(i)+1,q_2)\in D(w)$ with $q_2-q_1>1$ and no other empty boxes between them, as shown in Figure \ref{fig:k boxes not consecutive}. Then there is a pipe $i'$  with $q_1<i'<q_2$ and $w^{-1}(i')<w^{-1}(i)$, which implies $D(w)$ contains $D(31524)$. 
This completes the proof.
\end{proof}	
	
	\begin{figure}[ht]
		\centering
		\begin{tikzpicture}[scale=0.5]
		\draw[very thin] (0,0) rectangle (6,6); 
		\draw[fill=lightgray](1,2) rectangle (2,3);
  	\draw[fill=lightgray](2,2) rectangle (3,3);
		\draw[fill=lightgray](4,2) rectangle (5,3);
		\draw[thick, red](3.5,0)--(3.5,5.5)--(6,5.5);
		\draw[thick, blue](0.5,0)--(0.5,3.5)--(6,3.5);
		\node[below] at (0.5,0) {$i$};
		\node[below] at (3.5,0) {$i'$};
		\node[below] at (2.5,0) {$q_1$};
		\node[below] at (4.5,0) {$q_2$};	
		\node[right] at (6,5.5) {\footnotesize$w^{-1}(i')$};
		\node[right] at (6,3.5) {\footnotesize$w^{-1}(i)$};
		\node[right] at (6,2.5) {\footnotesize{$w^{-1}(i)$+$1$}};
		\end{tikzpicture}
		\caption{Non-consecutive $k$ boxes in row $w^{-1}(i)$+$1$}
		\label{fig:k boxes not consecutive}
	\end{figure}
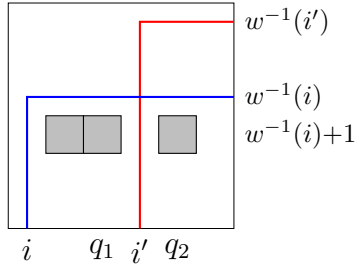

\begin{prop}\label{lem:factorization of Grothendieck polynomial when avoids 6 patterns}
Assume that $w$ avoids the patterns $1432$, $1342$, $13254$, $31524$, $12534$ and $21534$. Denote $\lambda$ by the northwest partition in $D(w)$. Then $\widetilde{\G}_{w}(x)$ is the product of $x^{\lambda}$ and a sequence of factors.  Further, the variables of each factor are disjoint. In other words, the variable $x_k$ appears in at most one factor for each $k$. 
\end{prop}
\begin{proof}
    Suppose the empty boxes in the northwest corner of $D(w)$ consist of a partition $\lambda$. Since no droops can be applied to the empty boxes in the northwest corner,  $x^{\lambda}$ is a common divisor of every monomial in $\widetilde{\G}_{w}(x)$. 

Since the empty boxes in each SE elbow satisfy Proposition \ref{prop:description of the Rothe pipe dream when avoiding 6 patterns} when $w$ avoids the six patterns, there are only two kinds of local structures on which we can apply droop operations. More precisely, if there is more than one box in the southeast region of a pipe, then $D(w)$ has local structure $A$  as in Figure \ref{fig:structure1}. If there is only one box to the southeast of a pipe, then $D(w)$ has local structure $B$ as in Figure \ref{fig:structure2}.
Note that in local structure $B$, we obtain a dominant permutation (i.e., 132-avoiding) by restricting $w$ to the pipes $j,\ldots,j+l-1$. 
Besides, there are no other boxes in the southeast region of the pipes $j,\ldots,j+l-1$.

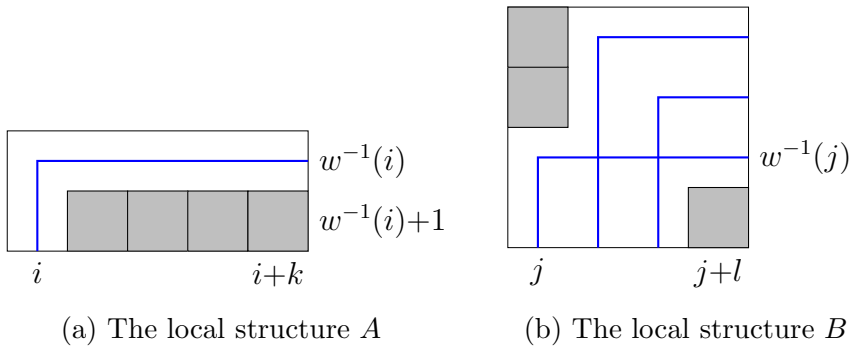
\begin{figure}[ht]
\centering
\begin{minipage}[t]{0.36\textwidth}
\begin{tikzpicture}[scale=0.8]
\draw[very thin](0,0) rectangle (5,2);
			\draw[blue,thick](0.5,0)--(0.5,1.5)--(5,1.5);
			\draw[fill=lightgray](1,0) rectangle (2,1);
			\draw[fill=lightgray](2,0) rectangle (3,1);
			\draw[fill=lightgray](3,0) rectangle (4,1);
			\draw[fill=lightgray](4,0) rectangle (5,1);
			\node[below] at (0.5,0){$i$};
			\node[below] at (4.5,0){$i$+$k$};
			\node[right] at (5,0.5){$w^{-1}(i)$+$1$};
			\node[right] at (5,1.5){$w^{-1}(i)$};
			\end{tikzpicture}
			\subcaption{The local structure $A$}
			\label{fig:structure1}
		\end{minipage}
\qquad
       \begin{minipage}[t]{0.3\textwidth}
			\begin{tikzpicture}[scale=0.8]
			\draw[very thin](0,0) rectangle (4,4);
			\draw[blue,thick](0.5,0)--(0.5,1.5)--(4,1.5);
			\draw[blue,thick](1.5,0)--(1.5,3.5)--(4,3.5);
			\draw[blue,thick](2.5,0)--(2.5,2.5)--(4,2.5);
			\draw[fill=lightgray](3,0) rectangle (4,1);
            \draw[fill=lightgray](0,2) rectangle (1,3);
            \draw[fill=lightgray](0,3) rectangle (1,4);
			\node[below] at (0.5,0) {$j$};
			\node[below] at (3.5,0) {$j$+$l$};
            \node[right] at (4,1.5) {$w^{-1}(j)$};
			\end{tikzpicture}
			\subcaption{The local structure $B$}
			\label{fig:structure2}
		\end{minipage}
		\caption{Two local structures}
		\label{fig:two local structures can apply droops}
\end{figure}

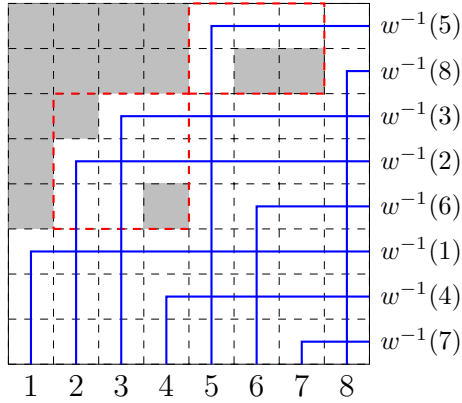
\begin{figure}[ht]
		\centering
		\begin{tikzpicture}[scale=0.6]
		\draw[very thin] (0,0) rectangle (8,8); 
        \draw[draw=none, fill=lightgray] (5,6) rectangle (7,7);
        \draw[draw=none, fill=lightgray] (3,3) rectangle (4,4);
        \draw[draw=none, fill=lightgray] (0,3)--(0,8)--(4,8)--(4,6)--(2,6)--(2,5)--(1,5)--(1,3)--(0,3);
		\draw[very thin, dashed, step=1] (0,0) grid (8,8); 
		\draw[thick,dashed,red](1,3) rectangle (4,6);
		\draw[thick,dashed,red](4,6) rectangle (7,8);

		\draw[thick, blue](0.5,0)--(0.5,2.5)--(8,2.5);
		\draw[thick, blue](1.5,0)--(1.5,4.5)--(8,4.5);
		\draw[thick, blue](2.5,0)--(2.5,5.5)--(8,5.5);
		\draw[thick, blue](3.5,0)--(3.5,1.5)--(8,1.5);
		\draw[thick, blue](4.5,0)--(4.5,7.5)--(8,7.5);
		\draw[thick, blue](5.5,0)--(5.5,3.5)--(8,3.5);
		\draw[thick, blue](6.5,0)--(6.5,0.5)--(8,0.5);
		\draw[thick, blue](7.5,0)--(7.5,6.5)--(8,6.5);
		\foreach \x [evaluate=\x as \label using int(\x)] in {1,2,...,8} {
			\node[below] at (\x-0.5,0) {\label};
		}
		
		\node [right] at (8,7.5){\footnotesize$w^{-1}(5)$};
		\node [right] at (8,6.5){\footnotesize$w^{-1}(8)$};
		\node [right] at (8,5.5){\footnotesize$w^{-1}(3)$};
		\node [right] at (8,4.5){\footnotesize$w^{-1}(2)$};
		\node [right] at (8,3.5){\footnotesize$w^{-1}(6)$};
		\node [right] at (8,2.5){\footnotesize$w^{-1}(1)$};
		\node [right] at (8,1.5){\footnotesize$w^{-1}(4)$};
		\node [right] at (8,0.5){\footnotesize$w^{-1}(7)$};
		\end{tikzpicture}
		\caption{$D(w)$ for $w=58326147$}
		\label{fig:Rothe pipe dream of 58326147}
	\end{figure}

Moreover, by Proposition \ref{prop:description of the Rothe pipe dream when avoiding 6 patterns}, the two kinds of local structures must lie in disjoint rows and columns, and arranged anti-diagonally along the southeast boundary of the partition $\lambda$. For an illustration, see Figure \ref{fig:Rothe pipe dream of 58326147} for $w=58326147$, two local structures are bounded by red dashed lines. Thus the droop operations on each local structure are independent of each other. Each local structure with its droop operations contributes to a factor of $\widetilde{\G}_{w}(x)$. The variables of each factor are disjoint since these local structures and their droop operations are in disjoint rows. 
\end{proof}

\begin{thm}\label{thm:factorization of Grothendieck polynomial when avoiding 6 patterns}
Assume that  $w$ avoids patterns $1432$, $1342$, $13254$, $31524$, $12534$ and $21534$. Denote $\lambda$ by the northwest partition in $D(w)$.  Then $\widetilde{\G}_{w}(x)$ is the product of $x^{\lambda}$ and factors in the following forms: 
\begin{align}
\widetilde{F}_k(x)&=\sum_{t=0}^{k}x_{p}^{t}x_{p+1}^{k-t}+\sum_{s=1}^{k}x_{p}^{s}x_{p+1}^{k+1-s},\label{eq:F factor}\\
\widetilde{G}_l(x)&=\sum_{s=1}^{l+1}e_{s}(x_r,\ldots,x_{r+l}).\label{eq:G factor}
\end{align}
Moreover, the variables in each factor are disjoint. 
\end{thm}

\begin{proof}
We compute the explicit expressions of the factors given by the two local structures in Figure \ref{fig:two local structures can apply droops}. For the local structure $A$, we can droop the SE elbow at $(w^{-1}(i),i)$ to each of the $k$ empty boxes in row $w^{-1}(i)+1$. Let $p=w^{-1}(i)$. Then the polynomial generated by the local structure $A$ is  
\begin{align*} 
\widetilde{F}_{k}(x)&=x_{p+1}^{k}+(1+x_{p+1})\left(\sum_{t=1}^{k}x_{p}^{t}x_{p+1}^{k-t}\right)\nonumber\\
	&=\sum_{t=0}^{k}x_{p}^{t}x_{p+1}^{k-t}+\sum_{s=1}^{k}x_{p}^{s}x_{p+1}^{k+1-s}.
\end{align*}

For the local structure $B$, let $\widehat{w}$ be the permutation obtained from $w$ by rearranging $w^{-1}(j),\ldots,w^{-1}(j+l-1)$ increasingly. 
By  \cite[Proposition 3.9]{MészárosSetiabrataDizier2021}, 
\begin{align}\label{doubledominat}
  \widetilde{\G}_w(x)=x^{\mu}\cdot\widetilde{\G}_{\widehat{w}}(x),  
\end{align}
where $\mu$ is the partition in the northwest corner of the local structure $B$.  
Thus we only need to compute the polynomial generated by the local structure $B'$ in Figure \ref{fig:Rearrangement of structure $B$}. 
For the running example $w=58326147$ in Figure \ref{fig:Rothe pipe dream of 58326147}, we have $j=l=2$, and so $\widehat{w}=58236147$. 

\begin{figure}[ht]
	\centering
        \begin{minipage}{0.3\textwidth}
			\begin{tikzpicture}[scale=0.8]
			\draw[very thin](0,0) rectangle (4,4);
			\draw[blue,thick](0.5,0)--(0.5,1.5)--(4,1.5);
			\draw[blue,thick](1.5,0)--(1.5,3.5)--(4,3.5);
			\draw[blue,thick](2.5,0)--(2.5,2.5)--(4,2.5);
			\draw[fill=lightgray](3,0) rectangle (4,1);
            \draw[fill=lightgray](0,2) rectangle (1,3);
            \draw[fill=lightgray](0,3) rectangle (1,4);
			\node[below] at (0.5,0) {$j$};
			\node[below] at (3.5,0) {$j$+$l$};
            \node[right] at (4,1.5) {$w^{-1}(j)$};
			\end{tikzpicture}
		\end{minipage}
$\rightarrow$
\quad
\begin{minipage}{0.3\textwidth}
    	\begin{tikzpicture}[scale=0.8]
	\draw[very thin](0,0) rectangle (4,4);
	\draw[blue,thick](0.5,0)--(0.5,3.5)--(4,3.5);
	\draw[blue,thick](1.5,0)--(1.5,2.5)--(4,2.5);
	\draw[blue,thick](2.5,0)--(2.5,1.5)--(4,1.5);
	\draw[fill=lightgray](3,0) rectangle (4,1);
	\node[below] at (0.5,0) {$j$};
	\node[below] at (3.5,0) {$j$+$l$};
	\node[right] at (4,3.5) {$\widehat{w}^{-1}(j)$};
	\node[right] at (4,0.5) {$\widehat{w}^{-1}(j)$+$l$};
	\end{tikzpicture}	
\end{minipage}
\caption{The local structure $B'$}
	\label{fig:Rearrangement of structure $B$}
\end{figure}
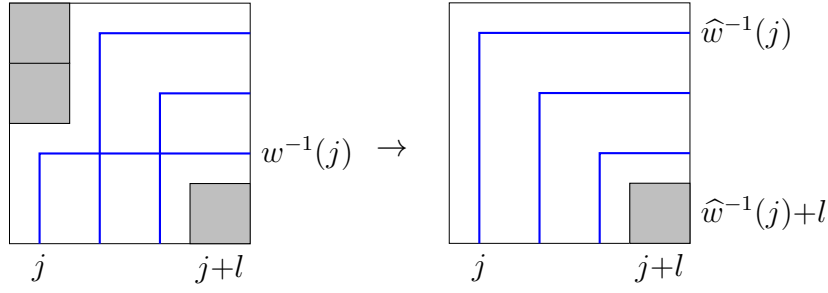

Assume that the local structure $B'$ is in the rectangle $[\widehat{w}^{-1}(j),\widehat{w}^{-1}(j)+l]\times[j,j+l]$. Let $r=\widehat{w}^{-1}(j)$. 
We claim that the polynomial generated by the local structure $B'$ is 
\begin{align*} 
\sum_{s=1}^{l+1}e_{s}(x_r,\ldots,x_{r+l}),
\end{align*}
where $e_{s}(x_r,\ldots,x_{r+l})$ is the degree $s$ elementary symmetric polynomial in variables $x_r,\ldots,$ $x_{r+l}$. 

To prove the claim, first droop  pipe $j+l-1$ to the empty box at $(r+l,j+l)$ and turn $(r+l-1,j+l-1)$ into a new empty box, this generates $x_{r+l-1}(1+x_{r+l})$. Then droop pipe $j+l-2$ to the new empty box and so on. In this process, the original empty box at $(r+l,j+l)$ moves northwest along the main diagonal and there are no $K$-theoretic droops. 
Therefore, all the monomials in the series of droops generated are
\[
x_{r+l}+x_{r+l-1}(1+x_{r+l})+\cdots+x_{r}(1+x_{r+1})\cdots (1+x_{r+l}),
\]
which equals $e_1(x_r,\ldots,x_{r+l})+\cdots+e_{l+1}(x_r,\ldots,x_{r+l})$, as desired.
\end{proof}

\begin{exmp}
For the running example $w=58326147$, the northwest partition is $\lambda=(4,4,2,1,1)$. One can check that $\widetilde{\G}_{58326147}(x)$ can be decomposed as follows: 
$$\widetilde{\G}_{58326147}(x)=x^{\lambda}(x_1^{2}+x_1x_2+x_2^{2}+x_1x_2^{2}+x_1^{2}x_2)(x_3+x_4+x_5+x_3x_4+x_4x_5+x_3x_5+x_3x_4x_5).$$
	The two factors are exactly the polynomials generated by the two local structures in the red dashed lines in Figure \ref{fig:Rothe pipe dream of 58326147}. 
\end{exmp}

\begin{proof}[Proof of Theorem \ref{thm:sufficiency}]
 By Theorem \ref{thm:factorization of Grothendieck polynomial when avoiding 6 patterns}, since the monomials in each factor of $\widetilde{\G}_{w}(x)$ are zero-one and variables in each factor are disjoint, we conclude that each monomial in $\widetilde{\G}_{w}(x)$ appears exactly once. 
\end{proof}

\section{Properties of zero-one Grothendieck polynomials}\label{sec:Properties of zero-one Grothendieck polynomials}

In this section, we apply the factorization of zero-one Grothendieck polynomials in Theorem \ref{thm:factorization of Grothendieck polynomial when avoiding 6 patterns} to prove  Theorem \ref{thm:conjectures of Lorentzian property2} and
Theorem  \ref{thm:conjectures of support and coefficient}.

\begin{proof}[Proof of Theorem \ref{thm:conjectures of Lorentzian property2}]
Recall that $\widetilde{\S}_{w}(x;y)$ is the lowest degree component of $\widetilde{\G}_{w}(x;y)$.   To obtain $\widetilde{\S}_{w}(x;y)$, by \eqref{eq:doubleG}, we only need to compute the weighted sum of BPDs with no two pipes cross more than once, and take the weight of an empty box at $(i,j)$ as $x_i+y_j$ and an NW elbow as $1$.

To obtain a factorization of $\widetilde{\S}_w(x;y)$, we need to explain that \eqref{doubledominat} also holds for double Schubert polynomials.  That is,
\[
\widetilde{\S}_w(x;y)=\widetilde{\S}_{\widehat{w}}(x;y)\cdot\prod_{(i,j)\in\mu}(x_i+y_j). 
\]
Notice that $\widetilde{\S}_w(x;y)$ can be generated by applying ladders moves to the bottom rc-graph of $w$, see \cite{FominKirillov1996} and \cite{BergeronBilley1993}. It is easy to see that the northwest partition $\mu$ does not affect the ladder moves.
Therefore,
  we have the factorization  $$\widetilde{\S}_{w}(x;y)=\prod_{(i,j)\in \lambda}(x_i+y_j)\prod_{k}\widetilde{F}_{k}(x;y)\prod_{l}\widetilde{G}_{l}(x;y),$$
where $\lambda$ is the partition in the northwest corner of $D(w)$, $\widetilde{F}_{k}(x;y)$ and $\widetilde{G}_{l}(x;y)$ are generated by the local structure $A$ and the local structure $B'$, respectively.

For the local structure $A$ in Figure \ref{fig:structure1}, 
it is easy to see that $\widetilde{F}_{k}(x;y)$ is eqaul to the double Schubert polynomial $\widetilde{\S}_u(x;y)$ with $u=s_{2}s_{3}\cdots s_{k+1}$ in variables $x_p,x_{p+1},y_i,\ldots,y_{i+k}$.
By \cite[eq. (6.3)]{Macdonald}, we have
\begin{align*}
\widetilde{F}_{k}(x;y)&=\sum_{u_1^{-1}u_2=u\atop \ell(u_1)+\ell(u_2)=\ell(u)}\S_{u_1}(x)\S_{u_2}(y)\\
  &=\sum_{t=0}^{k}h_{k-t}(x_{p},x_{p+1})e_{t}(y_i,\ldots,y_{i+k}).  
\end{align*}
It is obvious that the factor generated by the local structure $B'$ in Figure \ref{fig:Rearrangement of structure $B$} is 
$$\widetilde{G}_{l}(x;y)=x_{r}+\cdots+x_{r+l}+y_{j}+\cdots+y_{j+l}.$$

By \cite[Corollary 3.8]{PetterHuh2020}, if $N(f)$ and $N(g)$ are Lorentzian polynomials, then $N(fg)$ is a Lorentzian polynomial.  
It is easy to see that $N(\prod_{(i,j)\in \lambda}(x_i+y_j))$ and  $N(\widetilde{G}_{l}(x;y))$ are Lorentzian. In the following, we show that $N(\widetilde{F}_{k}(x;y))$ is Lorentzian. By \cite{CastilloCidRuizMohammadi2023}, double Schubert polynomials have SNP and whose Newton polytopes are generalized permutahedra, thus $\supp(\widetilde{F}_{k}(x;y))$ is M-convex.

Suppose that we take  $t$ derivatives on the $y$ variables, say, for example, on $y_{i},\ldots,y_{i+t-1}$. And take  $t_1$ and $t_2$ derivatives respectively on $x_p$ and $x_{p+1}$ with $t_1+t_2=k-2-t$. Then we have 
\begin{align*}
    &\left(\frac{\partial}{\partial x_{p}}\right)^{t_1} \left(\frac{\partial}{\partial x_{p+1}}\right)^{t_2}\frac{\partial}{\partial y_{i}}\cdots \frac{\partial}{\partial y_{i+t-1}}N(\widetilde{F}_{k}(x;y))\\
    &=\left(\frac{\partial}{\partial x_{p}}\right)^{t_1} \left(\frac{\partial}{\partial x_{p+1}}\right)^{t_2}\frac{\partial}{\partial y_{i}}\cdots \frac{\partial}{\partial y_{i+t-1}}\left(e_{t}(y)h_{k-t}(x)+e_{t+1}(y)h_{k-t-1}(x)+e_{t+2}(y)h_{k-t-2}(x)\right)\\
    &=\frac{1}{2}x_{p}^{2}+x_{p}x_{p+1}+\frac{1}{2}x_{p+1}^{2}+(x_{p}+x_{p+1})e_{1}(y_{i+t},\ldots,y_{i+k})+e_{2}(y_{i+t},\ldots,y_{i+k})\\
    		&=
		\begin{pmatrix}
		x_{p} &  x_{p+1} & y_{i+t} & \cdots & y_{i+k}
		\end{pmatrix}
		\begin{pmatrix}
		\frac{1}{2} & \frac{1}{2} & \cdots & \frac{1}{2} & \frac{1}{2}\\[5pt]
		\frac{1}{2} & \frac{1}{2} & \cdots & \frac{1}{2} & \frac{1}{2}\\[5pt]
		\vdots & \vdots &\ddots &\vdots & \vdots\\[5pt]
		\frac{1}{2} & \frac{1}{2} & \cdots & 0 & \frac{1}{2}\\[5pt]
		\frac{1}{2} & \frac{1}{2} & \cdots & \frac{1}{2} & 0\\
		\end{pmatrix}
		\begin{pmatrix}
		x_{p}\\[5pt]
            x_{p+1}\\[5pt]
            y_{i+t}\\[5pt]
		\vdots\\[5pt]
		y_{i+k}\\[5pt]
		\end{pmatrix}. 
    \end{align*}
The quadratic form is a $(k-t+3)\times (k-t+3)$ matrix, which has exactly $k-t+1$ zeros on the main diagonal except the first two $\frac{1}{2}$. One can check that the quadratic form  has characteristic polynomial
$$\lambda\left(\lambda+\frac{1}{2}\right)^{k-t}\left(\lambda^{2}-\frac{n-1}{2}\lambda-\frac{1}{2}\right),$$
which has exactly one positive eigenvalue. This completes the proof.
\end{proof}

\begin{proof}[Proof of Theorem \ref{thm:conjectures of support and coefficient}]
	Suppose that $w$ avoids the six patterns in Theorem \ref{thm:main theorem} and $\lambda$ is the northwest partition of  $D(w)$. Then $\widetilde{\G}_{w}(x)$ can be decomposed as 
	$$\widetilde{\G}_{w}(x)=x^{\lambda}\cdot \prod_{k} \widetilde{F}_{k}(x)\cdot \prod_{l}\widetilde{G}_{l}(x),$$
where $\widetilde{F}_{k}(x)$ and $\widetilde{G}_{l}(x)$ are factors in \eqref{eq:F factor} and \eqref{eq:G factor}    with disjoint variables.  
 
For Conjecture \ref{conj:support2}, suppose that $\alpha=\lambda+\sum_{k}\alpha_{k}+\sum_{l}\alpha_{l}\in \supp(\G_w)$ such that $|\alpha|< deg(\G_w)$, where $\alpha_{k}\in \supp(\widetilde{F}_{k}(x))$ and $\alpha_{l}\in \supp(\widetilde{G}_{l}(x))$. Then either $|\alpha_k|<deg(\widetilde{F}_{k}(x))$ or $|\alpha_l|<deg(\widetilde{G}_{l}(x))$.

If $|\alpha_k|<deg(\widetilde{F}_{k}(x))$, then $x^{\alpha_{k}}=x_{p}^{t}x_{p+1}^{k-t}$ for some $t$. Let $x^{\beta_{k}}$ be $x_{p}^{t+1}x_{p+1}^{k-t}$ if $t=0$, and $x_{p}^{t}x_{p+1}^{k-t+1}$ if $t>0$. Thus $\alpha_k<\beta_k$ and $|\beta_k|=|\alpha_k|+1$.
Similarly, if $|\alpha_l|<deg(\widetilde{G}_{l}(x))$, then it is easy to see there exists some  $\beta_{l}$ with $\alpha_{l}<\beta_{l}$ and $|\beta_{l}|=|\alpha_{l}|+1$.
This finishes the proof of Conjecture \ref{conj:support2}.

For Conjecture \ref{conj:coefficient}, 
suppose that $\G_{w}(x)$ can be decomposed as
$$\G_{w}(x)=x^{\lambda}\cdot \prod_{k}F_{k}(x)\prod_{l}G_{l}(x).$$
Since we need to consider the signs of coefficients, the factors $F_{k}(x)$ and $G_{l}(x)$ here are of the following forms
\begin{align*}
    F_{k}(x)&=\sum_{t=0}^{k}x_{p}^{t}x_{p+1}^{k-t}-\sum_{s=1}^{k}x_{p}^{s}x_{p+1}^{k+1-s},\\
 G_{l}(x)&=\sum_{s=1}^{l+1}(-1)^{s-1}e_{s}(x_r,\ldots,x_{r+l}).
\end{align*}

Let $\beta=\lambda+\sum_{k}\beta_{k}+\sum_{l}\beta_{l}\in \supp(\G^{top}_w(x))$. Since $x^{\beta}$ is a monomial of the highest degree, $x^{\beta_{k}}$ and $x^{\beta_{l}}$ are the highest degree monomials in each factor.  Since variables are disjoint in each factor, we have 
\begin{align}\label{eq:conjectures on coefficients}
\sum_{\alpha\leq\beta}C_{w\alpha}x^{\alpha}=x^{\lambda}\cdot\prod_{k}\left(\sum_{\alpha_{k}\leq\beta_{k}}f_{\alpha_{k}}x^{\alpha_{k}}\right)\prod_{l}\left(\sum_{\alpha_{l}\leq\beta_{l}}g_{\alpha_{l}}x^{\alpha_{l}}\right).
\end{align}
where $f_{\alpha_{k}}$ and $g_{\alpha_{l}}$ are respectively coefficients of $x^{\alpha_{k}}$ and $x^{\alpha_{l}}$ in $F_{k}(x)$ and $G_{l}(x)$, both equal to $\pm1$ since each factor is zero-one. Besides, the coefficients in $F_{k}(x)$ and $G_{l}(x)$ alternate in sign with degree, and the lowest terms are positive. 
	
	Suppose $\beta_{k}\in \supp(F_{k}(x))$ with
	$$\beta_{k}=(0,\ldots,s,k+1-s,0,\ldots,0)=s\epsilon_{p}+(k+1-s)\epsilon_{p+1},\text{ where }1\leq s \leq k.$$
Then the only three entries $\leq\beta_{k}$ in $\supp(F_{k}(x))$ are $\beta_{k}$, $\beta_{k}-\epsilon_{p}$ and $\beta_{k}-\epsilon_{p+1}$. 
Since the unique monomial  of highest degree in $G_{l}(x)$ is  $(-1)^{l+1}x_r\cdots x_{r+l}$, whose support is
	$$\beta_{l}=(0,\ldots,0,1,\ldots,1,0,\ldots,0).$$
    We have $\mu\leq\beta_{l}$ for all $x^{\mu}$ in $G_{l}(x)$. 
	
Set $x=(1,\ldots,1)$ in equation \eqref{eq:conjectures on coefficients}, for each $k$ and $l$ we have 
\begin{align*}
\sum_{\alpha_{k}\leq\beta_{k}}f_{\alpha_{k}}&=f_{\beta_{k}-\epsilon_{p}}+f_{\beta_{k}-\epsilon_{p+1}}+f_{\beta_{k}}=1+1-1=1,\\
\sum_{\alpha_{l}\leq\beta_{l}}g_{\alpha_{l}}&=G_{l}(1,\ldots,1)=\binom{l+1}{1}-\binom{l+1}{2}+\cdots+(-1)^{l+1}\binom{l+1}{l+1}=1.    
\end{align*}
Thus we conclude $$\sum_{\alpha\leq\beta}C_{w\alpha}=1,$$ 
as required.
\end{proof}

\footnotesize{

Y.M. C{\scriptsize HEN}, D{\scriptsize EPARTMENT OF} M{\scriptsize ATHEMATICS}, S{\scriptsize ICHUAN} U{\scriptsize NIVERSITY}, C{\scriptsize HENGDU} 610064, P.R. C{\scriptsize HINA.} Email address: 2020141210003@stu.scu.edu.cn

N{\scriptsize EIL}. J.Y. F{\scriptsize AN}, D{\scriptsize EPARTMENT OF} M{\scriptsize ATHEMATICS}, S{\scriptsize ICHUAN} U{\scriptsize NIVERSITY}, C{\scriptsize HENGDU} 610064, P.R. C{\scriptsize HINA.} Email address: fan@scu.edu.cn 

Z.L. Y{\scriptsize E}, S{\scriptsize CHOOL OF} M{\scriptsize ATHEMATICAL} S{\scriptsize CIENCE}, F{\scriptsize UDAN} U{\scriptsize NIVERSITY}, S{\scriptsize HANGHAI} 200433, P.R. C{\scriptsize HINA.}  Email address: 20307110133@fudan.edu.cn

}

\end{document}